\documentclass[11pt,leqno]{amsart}
\usepackage{epsfig}
\usepackage{amssymb}
\usepackage{amscd}
\usepackage[matrix,arrow]{xy}
\usepackage{graphicx}

\newtheorem{theo}{Theorem}[section]
\newtheorem{lemma}[theo]{Lemma}
\newtheorem{defi}[theo]{Definition}
\newtheorem{prop}[theo]{Proposition}

\newtheorem{remark}[theo]{Remark}

\newtheorem{example}[theo]{Example}
\numberwithin{equation}{section}

\def\bL{\mathbb{L}}

\def\Z{\mathbb{Z}}
\def\Q{\mathbb{Q}}

\def\P{{\mathcal{P}}}

\def\bR{{\mathbf R}}
\def\bL{{\mathbf L}}

\def\pre-tr{\operatorname{pre-tr}}

\def\Hom{\operatorname{Hom}}
\def\End{\operatorname{End}}

\def\Pic{\operatorname{Pic}}

\DeclareMathOperator*{\hocolim}{hocolim}

\newcommand{\QCoh}{\operatorname{QCoh}}
\newcommand{\Coh}{\operatorname{Coh}}

\newcommand{\Pol}{\operatorname{Pol}}
\newcommand{\pol}{\operatorname{pol}}

\newcommand{\cQ}{{\mathcal Q}}
\newcommand{\cF}{{\mathcal F}}

\newcommand{\cO}{{\mathcal O}}
\newcommand{\cP}{{\mathcal P}}

\newcommand{\cM}{{\mathcal M}}
\newcommand{\cN}{{\mathcal N}}
\newcommand{\cD}{{\mathcal D}}

\newcommand{\cA}{{\mathcal A}}
\newcommand{\cB}{{\mathcal B}}
\newcommand{\cI}{{\mathcal I}}
\newcommand{\cC}{{\mathcal C}}
\newcommand{\cE}{{\mathcal E}}

\newcommand{\cT}{{\mathcal T}}

\newcommand{\cHom}{\mathcal Hom}

\newcommand{\la}{\langle}
\newcommand{\ra}{\rangle}

\newcommand{\Fun}{\operatorname{Fun}}

\newcommand{\Perf}{\operatorname{Perf}}
\newcommand{\PsPerf}{\operatorname{PsPerf}}

\newcommand{\im}{\operatorname{Im}}

\newcommand{\rk}{\operatorname{rk}}

\newcommand{\Rep}{\operatorname{Rep}}

\newcommand{\Mat}{\operatorname{Mat}}

\newcommand{\Ext}{\operatorname{Ext}}

\newcommand{\add}{\operatorname{add}}

\newcommand{\Sym}{\operatorname{Sym}}

\newcommand{\Spec}{\operatorname{Spec}\,}

\newcommand{\id}{\operatorname{id}}

\newcommand{\Gr}{\operatorname{Gr}}

\usepackage{epsf}
\usepackage{amscd}

\newcommand{\mS}{\mathfrak{S}}

\title[Derived categories of Grassmannians and modular representation theory]
{Derived categories of Grassmannians over integers and modular representation theory}

\author{Alexander I. Efimov}
\address{Steklov Mathematical Institute of RAS, Gubkin str. 8, GSP-1, Moscow 119991, Russia}
\address{Laboratory of Algebraic Geometry, Higher School of Economics, 7 Vavilova str., Moscow, 117312, Russia}
\email{efimov@mccme.ru}
\thanks{MSC: 13D09, 14F05, 16G99.}
\thanks{The author was partially supported by RFBR grant 2998.2014.1 and RFBR research
project 15-51-50045.}

\begin{document}

\begin{abstract}In this paper we study the derived categories of coherent sheaves on Grassmannians $\Gr(k,n),$ defined over the ring of integers. We prove that the category $D^b(\Gr(k,n))$ has a semi-orthogonal decomposition, with components being full subcategories of the derived category of representations of $GL_k.$ This in particular implies existence of a full exceptional collection, which is a refinement of Kapranov's collection \cite{Kap}, which was constructed over a field of characteristic zero.

We also describe the right dual semi-orthogonal decomposition which has a similar form, and its components are full subcategories of the derived category of representations of $GL_{n-k}.$ The resulting equivalences between the components of the two decompositions are given by a version of Koszul duality for strict polynomial functors.

We also construct a tilting vector bundle on $\Gr(k,n).$ We show that its endomorphism algebra has two natural structures of a split quasi-hereditary algebra over $\Z,$ and we identify the objects of $D^b(\Gr(k,n)),$ which correspond to the standard and costandard modules in both structures.

All the results automatically extend to the case of arbitrary commutative base ring and the category of perfect complexes on the Grassmannian, by extension of scalars (base change). 

Similar results over fields of arbitrary characteristic were obtained independently in \cite{BLVdB}, by different methods.
\end{abstract}

\keywords{}

\maketitle

\tableofcontents

\section{Introduction}

In this paper we study the derived categories of Grassmannians over the ring of integers (and all the results automatically extend to the case of arbitrary commutative base ring).

Usually derived categories of coherent sheaves (or its smaller version, perfect complexes) are studied for algebraic varieties (or schemes) over fields, especially over the field of complex numbers. Sometimes the relative situation is considered, but again the base is usually a scheme over a field. However, most of the general notions (such as semi-orthogonal decompositions, tilting objects, exceptional collections) can be extended to the case of arbitrary basic ring. Moreover, given a scheme $Y,$ flat over $\Spec K,$ any result about the description of the (perfect) derived category of $Y$ immediately implies the corresponding result for $Y\times_K K',$ for any homomorphism $K\to K'$ to a commutative ring $K'.$

Fix some commutative base ring $K.$ Let $Y$ be a scheme which we assume smooth and proper over $\Spec K.$ Since $K$ may be non-coherent (hence non-noetherian), in general there is no abelian category of coherent sheaves on $Y.$ However, we always have a well defined triangulated category of  perfect complexes $\Perf(Y)\subset D(\QCoh Y)$ (which is exactly the subcategory of compact objects \cite{BVdB}).

\begin{remark}\label{remark:Perf=D^b}If $K$ is regular noetherian, then so is $Y,$ so in this case we have an equivalence $\Perf(Y)\simeq D^b(\Coh Y).$\end{remark}

A notion of exceptional object obviously extends to this setting: an object $\cE\in\Perf(Y)$ is called exceptional if $\bR\Hom(\cE,\cE)\cong K.$ Further, a sequence of exceptional objects $\cE_1,\dots,\cE_m\in \Perf(Y)$ is called an {\it exceptional collection} if $\bR\Hom(\cE_i,\cE_j)=0$ for $i>j.$ An exceptional collection in $\Perf(Y)$ is called full if it classically generates $\Perf(Y).$

For the definition of a tilting object to make sense, one additional assumption is needed.

\begin{defi}\label{def:tilting_object_intro}An object $\cE\in\Perf(Y)$ is called tilting if it is a generator, and satisfies the following properties:

(i) $\Hom^i(\cE,\cE)=0$ for $i\ne 0;$

(ii) the $K$-module $\Hom(\cE,\cE)$ is finitely generated projective.\end{defi}

The reason for adding an additional assumption (ii) in Definition \ref{def:tilting_object_intro} is the following: we want the class of tilting objects to be stable under base change. More precisely, if $\cE\in\Perf(Y)$ is a tilting object, and $K\to K'$ a homomorphism to a commutative ring $K',$ then the object $K'\stackrel{\bL}{\otimes}_K\cE\in\Perf(Y\times_K K')$ is also tilting.

Now let $X=\Gr(k,n)$ be the Grassmannian of $k$-dimensional subspaces in the $n$-dimensional space, where $0\leq k\leq n,$ defined over $K.$ Given a commutative $K$-algebra $R,$ the set $X(R)$ of $R$-points of $X$ is identified with the set of $R$-submodules $P\subset R^n$ such that the $R$-module $R^n/P$ is projective of constant rank $n-k$ (so that $P$ is projective of constant rank $k$). Clearly, $X$ is smooth and proper over $\Spec K.$ We have the tautological vector subbundle $\cF$ of rank $k,$ and the tautological quotient bundle $\cQ$ of rank $n-k$ on $X.$ They are related by a short exact sequence
$$0\to\cF\to\cO_X^{\oplus n}\to \cQ\to 0.$$

Recall that a Young diagram $\lambda$ is given by a non-increasing sequence of non-negative integers $\lambda_1\geq\lambda_2\geq\dots$ such that $\lambda_l=0$ for some $l.$ For non-negative integers $a,b\geq 0$ we denote by $\cP(a,b)$ the set of Young diagrams $\lambda$ such that $\lambda_1\leq a,$ $\lambda_{b+1}=0.$ The following result of Kapranov is well known.

\begin{theo}\label{th:kapranov}(\cite{Kap}) Suppose that $K$ is a field of characteristic zero. Then the category $D^b(X)$ has full strong exceptional collection $\{S_{\lambda}(\cF)\}_{\lambda\in\cP(n-k,k)}.$ Its right dual exceptional collection is $\{S_{\mu}(\cQ)[-|\mu|]\}_{\mu\in\cP(k,n-k)}.$ It is also strong.\end{theo}

Here $S_{\lambda}$ denotes the Schur functor associated to $\lambda.$

In this paper we consider the "universal case" $K=\Z.$ We construct a semi-orthogonal decomposition of $D^b(X),$ with all components having full exceptional collections, and describe the dual decomposition. Moreover, we construct a tilting bundle with nice properties of its endomorphism algebra.

Let us write simply $GL_k$ for the group $GL_k(\Z)$ considered as an algebraic group over $\Z.$ Denote by $\Rep(\Z,GL_k)$ the exact category of $GL_k$-modules which are free finitely generated over $\Z.$ We have a natural exact tensor functor $\Phi:\Rep(\Z,GL_k)\to \Coh X,$ sending tautological representation to $\cF.$ Let us denote by $\Rep(\Z,GL_k)^d\subset\Rep(\Z,GL_k)$ the subcategory of representations of degree $d.$ Here the degree is taken w.r.t. the center $G_m\subset GL_k.$ Further, for any integers $a\leq b$ we denote by $\Rep(\Z,GL_k)^d_{[a,b]}\subset \Rep(\Z,GL_k)^d$ the subcategory of representations for which all the weights $\lambda\in\Z^k$ satisfy $a\leq \lambda_i\leq b,$ for $1\leq i\leq k.$ We denote by  $\Phi^d_{[a,b]}$ the restriction of $\Phi$ to $\Rep(\Z,GL_k)^d_{[a,b]}.$

Our first main result is the following theorem (see Theorem \ref{th:SOD_of_Gr} for a more precise statement).

\begin{theo}\label{th:SOD_of_Gr_intro}The functors $\Phi^d_{[0,n-k]}:D^b(\Rep(\Z,GL_k)^d_{[0,n-k]})\to D^b(X)$ are fully faithful for $0\leq d\leq k(n-k).$ We have a semi-orthogonal decomposition
$$D^b(X)=\la \im(\Phi^{k(n-k)}_{[0,n-k]}),\im(\Phi^{k(n-k)-1}_{[0,n-k]}),\dots,\im(\Phi^{0}_{[0,n-k]})\ra.$$
\end{theo}

Further, we have a similar exact tensor functor $\Psi:\Rep(\Z,GL_{n-k})\to \Coh(X),$ which sends the tautological representation to $\cQ.$
We denote by $\Psi^d_{[a,b]}$ its restriction to $\Rep(\Z,GL_{n-k})^d_{[a,b]}.$ The next main result is the following theorem (see Theorem \ref{th:right_dual_decomposition}).

\begin{theo}\label{th:right_dual_decomposition_intro}The functors $\Psi^d_{[0,k]}:D^b(\Rep(\Z,GL_{n-k})^d_{[0,k]})\to D^b(X)$ are fully faithful and we have a semi-orthogonal decomposition
\begin{equation}\label{decomposition_Psi}D^b(X)=\la \im(\Psi^0_{[0,k]}),\im(\Psi^1_{[0,k]}),\dots,\im(\Psi^{k(n-k)}_{[0,k]})\ra,\end{equation}
which is right dual to the decomposition
$$D^b(X)=\la \im(\Phi^{k(n-k)}_{[0,n-k]}),\im(\Phi^{k(n-k)-1}_{[0,n-k]}),\dots,\im(\Phi^{0}_{[0,n-k]})\ra$$
from Theorem \ref{th:SOD_of_Gr_intro}.
\end{theo}

The resulting equivalence $D^b(\Rep(\Z,GL_k)^d_{[0,n-k]})\simeq D^b(\Rep(\Z,GL_{n-k})^d_{[0,k]})$ is shown to be a certain version of (inverse) Koszul duality functor for strict polynomial functors (see Proposition \ref{prop:koszul_for_GL}).

As a direct application, we get full exceptional collections on $X$ (see Theorem \ref{th:exceptional_on_Gr}).

\begin{theo}\label{th:exceptional_on_Gr_intro}1) The category $D^b(X)$ has a full exceptional collection $\{S_{\lambda}(\cF)\}_{\lambda\in\cP(n-k,k)}.$ Its right dual exceptional collection is $\{S_{\mu}(\cQ)[-|\mu|]\}_{\mu\in\cP(k,n-k)}.$

2) The category $D^b(X)$ has a full exceptional collection $\{W_{\lambda}(\cF)\}_{\lambda\in\cP(n-k,k)}.$ Its right dual exceptional collection is $\{W_{\mu}(\cQ)[-|\mu|]\}_{\mu\in\cP(k,n-k)}.$\end{theo}

Here $W_{\lambda}$ is a Weyl functor associated to $\lambda.$ In characteristic zero we have $S_{\lambda}=W_{\lambda},$ which agrees with Theorem \ref{th:kapranov}. We refer to Subsection \ref{ssec:Schur_algebras} for the definitions of $S_{\lambda}$ and $W_{\lambda}$ in the characteristic-free approach.

Another main result concerns the tilting vector bundle on $X.$  For a Young diagram, we put $$\Lambda^{\lambda}(\cF):=\bigotimes\limits_{i\geq 0}\Lambda^{\lambda_i}(\cF).$$  Consider the following vector bundle on $X:$
$$\cE(k,n)=\bigoplus\limits_{\lambda\in\cP(k,n-k)}\Lambda^{\lambda}(\cF).$$

We obtain the following result (see Theorem \ref{th:tilting_cE(k,n)_on_X}).

\begin{theo}\label{th:tilting_cE(k,n)_on_X_intro}The vector bundle $\cE(k,n)$ is a tilting object of $D^b(X).$\end{theo}

We refer to Subsection \ref{ssec:hw_def_and _properties} for the definitions and basic properties of split quasi-hereditary $K$-algebras and highest weight categories. Here we just mention that a split quasi-hereditary $K$-algebra $A$ is finitely generated projective as a $K$-module, and the triangulated category $\Perf(A)$ has two natural exceptional collections. One of them consists of the so-called {\it standard} $A$-modules, and its left dual consists of the {\it costandard} $A$-modules. The class of split quasi-hereditary algebras, as well as standard and costandard modules, is stable under extension of scalars.

Let us put
$$\cB(k,n):=\End(\cE(k,n)).$$ Our final main result is the following (see Theorem \ref{th:B(k,n)_quasi_hereditary}).

\begin{theo}\label{th:B(k,n)_quasi_hereditary_intro}The algebra $\cB(k,n)$ has two natural structures of a split quasi-hereditary algebra.

1) In the first structure, the standard (resp. costandard) $\cB(k,n)$-modules correspond under the equivalence $\Perf(\cB(k,n))\simeq\Perf(X)$ exactly to $S_{\lambda}(\cF)$ (resp. $S_{\mu}(\cQ)\otimes\omega_X[k(n-k)-|\mu|]$), where $\lambda\in\cP(n-k,k)$ (resp. $\mu\in\cP(k,n-k)$).

2) In the second structure, the standard (resp. costandard) $\cB(k,n)$-modules correspond under the equivalence $\Perf(\cB(k,n))\simeq\Perf(X)$ exactly to $W_{\mu}(\cQ)\otimes\omega_X[k(n-k)-|\mu|]$ (resp. $W_{\lambda}(\cF)\otimes\omega_X[k(n-k)]$), where $\mu\in\cP(k,n-k)$ (resp. $\lambda\in\cP(n-k,k)$).\end{theo}

The paper is organized as follows.

In Section \ref{sec:triangulated_DG} we recall the basic notions about triangulated and DG categories. In Subsection \ref{ssec:SOD} we recall semi-orthogonal decompositions and the notion of left and right dual decompositions. In Subsection \ref{ssec:smoothness_properness} we recall smooth and proper DG categories. In Subsection \ref{ssec:exceptional} we define the notion of exceptional collections in enhanced triangulated categories over a commutative ring, the notion of left and right dual exceptional collection, and the notion of a tilting object.

Section \ref{sec:highest_weight} is devoted to split quasi-hereditary algebras and highest weight categories over a commutative ring. In this version they were defined by Rouquier \cite{Ro}.  In Subsection \ref{ssec:hw_def_and _properties} we recall the definitions and basic properties, essentially all the results here are contained in \cite{Ro}. The slight difference here is that we consider the case of arbitrary basic commutative ring $K,$ while in \cite{Ro} $K$ is assumed to be noetherian. In particular for a finite projective $K$-algebra $A,$ instead of the category $\text{mod-}A$ of finitely generated $A$-modules (which may not be abelian) we consider exact category $\Rep(K,A)$ of right $A$-modules which are finitely generated projective over $K.$ Subsection \ref{ssec:gluing_of_quasi_her} is devoted to the gluing of split quasi-hereditary algebras via bimodules. Here we obtain a new result which states that split quasi-hereditary structures are preserved (at least in two natural ways) by gluing under some very natural assumption on bimodules (to be standardly filtered), very similar to the gluing of smooth DG categories.

Section \ref{sec:pol_functors_GL_n_Koszul} is devoted to strict polynomial functors, Schur algebras and representations of $GL_n.$ In Subsection \ref{ssec:pol_functors} we recall the category of strict polynomial functors of degree $d$ over a commutative ring $K,$ which was introduced by Friedlander and Suslin \cite{FS}. Here we follow Krause \cite{Kr}. We define the internal tensor product, internal $\Hom,$ and external tensor product. In section \ref{ssec:Schur_algebras} we recall the definition of Schur algebra $S_K(n,d).$ This algebra is split quasi-hereditary for all $n,d\geq 0,$ and for $n\geq d$ the category $S_K(n,d)\text{-Mod}$ is equivalent to the category of strict polynomial functors of degree $d.$ We recall Schur and Weyl functors, which are respectively costandard and standard objects in the category of strict polynomial functors. Further, we formulate the universal form of Littlewood-Richardson rule \cite{Bo}. In Subsection \ref{ssec:Koszul_duality} we define the Koszul duality functor on the derived category of strict polynomial functors, as well as its inverse. These functors are of the form $\bR\cHom_{\Gamma^d_K}(\Lambda^d,-)$ and $\Lambda^d\stackrel{\bL}{\otimes}_{\Gamma^d_K}-$ respectively. Here $-\otimes_{\Gamma^d_K}-$ states for the inner tensor product, and $\cHom_{\Gamma^d_K}(-,-)$ for the inner $\Hom.$ Finally, the subsection \ref{ssec:reps_of_GL_n} is devoted to the category $\Rep(K,GL_n)$ of representations of $GL_n$ over $K.$ We discuss a connection of this category with strict polynomial functors, and show that Koszul duality induces equivalences $D^b(\Rep(\Z,GL_n)^d_{[0,m]})\simeq D^b(\Rep(\Z,GL_{m})^d_{[0,n]})$ for non-negative integers $n,m$ (Proposition \ref{prop:koszul_for_GL}). Here we also recall a tilting object
in $D^b(\Rep(K,GL_n))^d_{[0,m]}.$

In section \ref{sec:base_change} we recall the basic fact about base change (Proposition \ref{prop:base_change}). It implies that essentially all the results about semi-orthogonal decompositions, tilting objects and exceptional collections for a quasi-compact quasi-separated scheme $X,$ flat over $\Spec K,$ are preserved under base change (extension of scalars) $K\to K'.$

In Section \ref{sec:D^b(Grassmannians)} we formulate and prove our main results about derived categories of coherent sheaves on the Grassmannian $X=\Gr(k,n)$ over integers. In Subsection \ref{ssec:SOD_of_Gr} we construct the semi-orthogonal decomposition (Theorem \ref{th:SOD_of_Gr}), with components being equivalent to $D^b(\Rep(\Z,GL_k)^d_{[0,n-k]}),$ $0\leq d\leq k(n-k).$ The proof uses GIT quotient presentation of Grassmannian, and Cousin-Grothendieck spectral sequence. In Subsection \ref{ssec:dual_decomposition_Koszul} we describe the dual decomposition (Theorem \ref{th:right_dual_decomposition}), with components being equivalent to $D^b(\Rep(\Z,GL_{n-k})^d_{[0,k]}).$ The resulting equivalence $D^b(\Rep(\Z,GL_k)^d_{[0,n-k]})\simeq D^b(\Rep(\Z,GL_{n-k})^d_{[0,k]})$ is just the version of inverse Koszul duality from Proposition \ref{prop:koszul_for_GL}. In Subsection \ref{ssec:tilting_vector_bundle} we prove Theorem \ref{th:tilting_cE(k,n)_on_X_intro} (this is Theorem \ref{th:tilting_cE(k,n)_on_X} below) and Theorem \ref{th:B(k,n)_quasi_hereditary_intro} (this is Theorem \ref{th:B(k,n)_quasi_hereditary} below).

{\noindent{\bf Acknowledgements.}} I am grateful to Alexander Kuznetsov, Dmitry Kaledin and Dmitri Orlov for useful discussions. The paper was prepared within the framework of a
subsidy granted to the HSE by the Government of the Russian Federation for the
implementation of the Global Competitiveness Program. 

The similar results were obtained by R.-O. Buchweitz, G.J. Leuschke and M. Van den Bergh \cite{BLVdB} over the field of arbitrary characteristic, by different methods (using in particular the Kempf vanishing theorem).

\section{Triangulated and DG categories.}
\label{sec:triangulated_DG}

\subsection{Semi-orthogonal decompositions}
\label{ssec:SOD}

For a class of object $\cC$ in a triangulated category $\cT,$ we use the standard notation for the left and right orthogonals to $\cC:$

$$\cC^{\perp}=\{X\in\cT\mid \Hom^{\bullet}(E,X)=0\text{ for any }E\in\cC\},$$
$${}^{\perp}\cC=\{X\in\cT\mid \Hom^{\bullet}(X,E)=0\text{ for any }E\in\cC\}.$$

\begin{defi}(\cite{BK}) Let $\cT$ be a triangulated category. A semi-orthogonal decomposition of $\cT$ is by definition a collection of full triangulated subcategories $\la\cA_1,\cA_2,\dots,\cA_n\ra,$ such that $\cA_i\subset\cA_j^{\perp}$ for $i<j,$ and $\cT$ is generated by $\cA_1,\dots,\cA_m$ as triangulated category.\end{defi}

\begin{defi}(\cite{BK},\cite{B}) Let $\cT$ be a triangulated category. A full triangulated subcategory $\cA\subset\cT$ is called left (resp. right) admissible if the inclusion functor $i:\cA\to\cT$ admits a left (resp.) right adjoint. A subcategory $\cA$ is called admissible if it is both left and right admissible.\end{defi}

\begin{lemma}(\cite{B}) For a semi-orthogonal decomposition $\cT=\la\cA,\cB\ra,$ the subcategory $\cA$ is left admissible and the subcategory $\cB$ is right admissible. Conversely, if $\cA\subset\cT$ is left (resp. right) admissible, then there is a semi-orthogonal decomposition $\cT=\la\cA,{}^{\perp}\cA\ra$ (resp. $\cT=\la\cA^{\perp},\cA\ra$).\end{lemma}

Given a semi-orthogonal decomposition $\cT=\la\cA,\cB\ra$ with $\cA$ admissible (which in this case is equivalent to right admissibility), we put
$$\bL_{\cA}\cB:=\cA^{\perp},$$
so that we have a semi-orthogonal decomposition $\cT=\la\bL_{\cA}\cB,\cA\ra.$ We also denote by $\bL_{\cA}$ the following composition of equivalences:
\begin{equation}\label{equivalence_L_A}\bL_{\cA}:\cB\stackrel{\sim}{\to}\cT/\cA\stackrel{\sim}{\to}\bL_{\cA}\cB.\end{equation}
Similarly, if $\cB$ is admissible (which in this case is equivalent to left admissibility) we put
$$\bR_{\cB}\cA:={}^{\perp}\cB,$$
\begin{equation}\label{equivalence_R_B}\bR_{\cB}:\cA\stackrel{\sim}{\to}\cT/\cB\stackrel{\sim}{\to}\bR_{\cB}\cA.\end{equation}

\begin{remark}If the triangulated subcategories $\cA,\cB\subset\cT$ are semi-orthogonal but do not generate $\cT,$ we may (and will) still consider mutations in the triangulated subcategory $\cT'\subset\cT,$ generated by $\cA$ and $\cB.$\end{remark}

\begin{defi}\label{def:dual_decompositions}Let $\cT=\la\cA_1,\dots,\cA_m\ra$ be an SOD.

1) Let us assume that for all $i$ the subcategory
$\cA_i$ is (right) admissible in $\la\cA_i,\dots,\cA_m\ra.$ Then the left dual SOD
$$\cT=\la\cB_m,\dots,\cB_1\ra$$ is defined by the formula
$$\cB_i:=\bL_{\cA_1}\bL_{\cA_2}\dots\bL_{\cA_{i-1}}\cA_i=\bL_{\la\cA_1,\dots,\cA_{i-1}\ra}\cA_i,\quad 1\leq i\leq m.$$

2) Let us assume that for all $i$ the subcategory
$\cA_i$ is (left) admissible in $\la\cA_1,\dots,\cA_i\ra.$ Then the right dual SOD
$$\cT=\la\cC_m,\dots,\cC_1\ra$$ is defined by the formula
$$\cC_i:=\bR_{\cA_m}\bR_{\cA_{m-1}}\dots\bR_{\cA_{i+1}}\cA_i=\bR_{\la\cA_{i+1},\dots,\cA_{m}\ra}\cA_i,\quad 1\leq i\leq m.$$\end{defi}

\begin{remark}Under the assumptions of Definition \ref{def:dual_decompositions} 1), the components of the left dual decomposition are determined by the equality
$$\cB_i:=\la\cA_1,\dots,\cA_{i-1},\cA_{i+1},\dots,\cA_{m}\ra^{\perp}.$$

Similarly, under the assumptions of Definition \ref{def:dual_decompositions} 2), the components of the right dual decomposition are determined by the equality
$$\cC_i:={}^{\perp}\la\cA_1,\dots,\cA_{i-1},\cA_{i+1},\dots,\cA_{m}\ra.$$

Moreover, the operations of taking left and right dual decompositions are mutually inverse.\end{remark}

\begin{prop}\label{prop:criterion_for_dual_decomposition}Let $\cT=\la\cA_1,\dots,\cA_m\ra$ be a semi-orthogonal decomposition, satisfying the assumptions of Definition \ref{def:dual_decompositions} 1).

i) Then the left dual semi-orthogonal decomposition $\cT=\la \cB_m,\dots,\cB_1\ra$ is determined uniquely by the following property:

 - for each $1\leq i\leq m,$ we have an equality of full triangulated subcategories $\la\cA_1,\dots,\cA_i\ra=\la \cB_{i},\dots,\cB_1\ra.$

ii) Denote by $F_i:\cA_i\stackrel{\sim}{\to}\cB_i$ the equivalence $\bL_{\la\cA_1,\dots,\cA_{i-1}\ra}$ (see \eqref{equivalence_L_A}). Then we have $$\Hom_{\cT}(X,Y)\cong\Hom_{\cB_i}(F_i(X),Y),\quad X\in\cA_i,\,Y\in\cB_i.$$

iii) Denote by $G_i:\cB_i\stackrel{\sim}{\to}\cA_i$ the equivalence $\bR_{\la\cB_{i-1},\dots,\cB_1\ra}$ (see \eqref{equivalence_R_B}). Then we have $$\Hom_{\cT}(X,Y)\cong\Hom_{\cA_i}(X,G_i(Y)),\quad X\in\cA_i,\,Y\in\cB_i.$$
\end{prop}

\begin{proof}This follows immediately from Definition \ref{def:dual_decompositions} 1).\end{proof}

\subsection{Smoothness and properness}
\label{ssec:smoothness_properness}

Let now $K$ denote the basic commutative ring. All DG categories will be assumed to be small $K$-linear. All triangulated categories from now on will be assumed to be Karoubi (idempotent) complete, and to be equipped with a DG enhancement (in particular, they themselves are $K$-linear). Sometimes we will tacitly identify a triangulated category with its DG enhancement.

All (DG) modules are assumed to be right. For a DG category $\cA$ we denote by $\Perf(\cA)$ the triangulated category of perfect DG $\cA$-modules.

\begin{defi}(\cite{TV}) 1) A DG category $\cA$ is called smooth over $K$ if the diagonal $\cA$-bimodule is perfect:
$$I_{\cA}\in\Perf(\cA\stackrel{\bL}{\otimes}_K\cA^{op}).$$

2) A DG category is called proper over $K$ if for any two objects $X,Y\in\cA$ we have
$$\Hom_{\cA}(X,Y)\in\Perf(K).$$\end{defi}

For convenience, we say that an enhanced triangulated category is smooth (resp. proper) if its DG enhancement is.

\begin{defi}(\cite{TV}) Let $\cA$ be a DG category. A DG $\cA$-module $M\in\text{Mod-}\cA$ is called pseudo-perfect if for any object $X\in\cA$ we have $M(X)\in\Perf(K).$ Pseudo-perfect $\cA$-modules form a triangulated subcategory of $D(\cA),$ which we denote by $\PsPerf(\cA).$\end{defi}

\begin{prop}\label{prop:Perf_vs_PsPerf}(\cite{TV}) If a DG category $\cA$ is smooth (resp. proper), then $\PsPerf(\cA)\subset\Perf(A)$ (resp. $\Perf(\cA)\subset\PsPerf(\cA)$).\end{prop}

\begin{prop}\label{prop:gluing_of_smooth_DG}(\cite{LS}) Suppose that a proper triangulated category $\cT$ has a semi-orthogonal decomposition: $\cT=\la\cA,\cB\ra$ (hence $\cA$ and $\cB$ are also proper). Then $\cT$ is smooth if and only if both $\cA$ and $\cB$ are smooth.\end{prop}

\subsection{Exceptional collections and tilting objects}
\label{ssec:exceptional}

\begin{defi} Let $\cT$ be a triangulated category. An object $E\in\cT$ is called exceptional if $$\Hom^n(E,E)=\begin{cases}K & \text{for }n=0;\\
0 & \text{otherwise.}\end{cases}$$\end{defi}

It follows from our assumptions that for an exceptional object $E\in\cT$ we have $$\la E\ra\simeq\Perf(K).$$

\begin{defi}Let $\cT$ be a proper triangulated category. An exceptional collection in $\cT$ is by definition a collection of exceptional objects $\la E_1,\dots,E_m\ra$ such that $\Hom^{\bullet}(E_i,E_j)=0$ for $i>j.$ An exceptional collection is full if it generates $\cT$ as a triangulated category.\end{defi}

Clearly, any full exceptional collection $\la E_1,\dots,E_m\ra$ provides a semi-orthogonal decomposition with components $\la E_i\ra,$ which are equivalent to $\Perf(K).$

IF $\la E, F\ra$ is a (not necessarily full) exceptional pair in a proper triangulated category $\cT,$ then the left and right mutations are defined by exact triangles:
$$L_E F[-1]\to \bR\Hom(E,F)\stackrel{\bL}{\otimes}_K E\to F\to \bL_E F,$$
$$\bR_F E\to E\to \bR\Hom(E,F)^{\vee}\stackrel{\bL}{\otimes}_K F\to\bR_F E[1].$$
It is easy to check that both $\la \bL_E F,E\ra$ and $\la F,\bR_F E\ra$ are still exceptional pairs, which generate the same triangulated subcategory of $\cT$ as $\la E,F\ra.$ Moreover, we have
$$\la\bL_E F\ra=\bL_{\la E\ra}\la F\ra,\quad \la\bR_F E\ra=\bR_{\la F\ra}\la E\ra.$$

\begin{defi}\label{def:dual_collections}Let $\la E_1,\dots,E_m\ra$ be an exceptional collection in a proper triangulated category $\cT.$

The left dual exceptional collection $\la E_m',\dots,E_1'\ra$ is defined by the formula
$$E_i':=\bL_{E_1}\bL_{E_2}\dots\bL_{E_{i-1}}E_i,\quad 1\leq i\leq m.$$

The right dual exceptional collection $\la E_m'',\dots,E_1''\ra$ is defined by the formula
$$E_i'':=\bR_{E_m}\bR_{E_{m-1}}\dots\bR_{E_{i+1}}E_i,\quad 1\leq i\leq m.$$
\end{defi}

\begin{prop}\label{prop:criterion_for_dual_collection}Let $\la E_1,\dots,E_m\ra$ be a full exceptional collection in a triangulated category $\cT.$ The left dual exceptional collection $\la E_m',\dots,E_1'\ra$ is uniquely determined by the following property:
$$\Hom^l(E_i,E_j')=\begin{cases}K & \text{for }l=0,\,i=j;\\
0 & \text{otherwise.}\end{cases}$$
Similarly, the right dual exceptional collection $\la E_m'',\dots,E_1''\ra$ is uniquely determined by the following property:
$$\Hom^l(E_i'',E_j)=\begin{cases}K & \text{for }l=0,\,i=j;\\
0 & \text{otherwise.}\end{cases}$$\end{prop}

We will also use the notion of an exceptional collection with a partial order.

\begin{defi}Let $\cT$ be a proper triangulated category, and $(\Delta\subset\cT,\preceq)$ a finite collection of exceptional objects, together with a partial order. We say that $\Delta$ is a partially ordered exceptional collection if for any $D_1,D_2\in\Delta$ we have $\Hom^{\bullet}(D_1,D_2)=0$ unless $D_1\preceq D_2.$\end{defi}

\begin{remark}Clearly, if $(\Delta\subset\cT,\preceq)$ is a partially ordered exceptional collection, and the order $\preceq'$ is a refinement of $\preceq,$ then $(\Delta,\preceq')$ is also a partially ordered exceptional collection. Moreover, any exceptional collection has a smallest partial order $\preceq,$ such that $\Hom^{\bullet}(D_1,D_2)\ne 0$ implies $D_1\preceq D_2.$\end{remark}

\begin{prop}\label{prop:order_on_dual_collection}Let $(\Delta,\preceq)$ be a full exceptional collection in a proper triangulated category $\cT.$ Choose some total (linear) refinement of the order $\preceq,$ and take the left dual exceptional collection $\Delta'$ (Definition \ref{def:dual_collections}). Then the set of objects $\Delta'$ does not depend on the order $\preceq$ and its refinement (up to isomorphism), and moreover $\Delta'$ is also exceptional w.r.t the partial order $\preceq'$ which corresponds to the opposite of $\preceq$ under the natural bijection $\Delta'\simeq\Delta.$

The same holds for the right dual exceptional collection.\end{prop}

\begin{proof}Indeed, by Proposition \ref{prop:criterion_for_dual_collection} the set $\Delta'$ depends only on $\Delta,$ not on the order $\preceq.$

Now we show that $\Delta'$ is exceptional w.r.t. $\preceq'.$ Denote by $\widetilde{\cT}$ the DG enhancement of $\cT,$ and $\cA\subset\widetilde{\cT}$ the full DG subcategory with the set of objects $\Delta.$ In particular, we have an equivalence $\cT\simeq\Perf(\cA).$

We may and will assume that $\Hom_{\cA}(D,D)=K$ for $D\in\Delta,$ and $\Hom_{\cA}(D_1,D_2)=0$ unless $D_1\preceq D_2.$ By Proposition \ref{prop:criterion_for_dual_collection}, the objects of the left dual exceptional collection corresponds to the perfect DG modules $M_D\in\text{Mod-}\cA,$ $D\in\Delta,$ such that $$M_D(E)=\begin{cases}K &\text{for }E=D;\\
0 &\text{for }E\in\Delta\setminus \{D\}.\end{cases}$$
Computting $\bR\Hom_{\cA}(M_{D_1},M_{D_2})$ via the bar resolution, we see that
$\bR\Hom_{\cA}(M_{D_1},M_{D_2})=0$ unless $D_1\succeq D_2.$\end{proof}

\section{Highest weight categories and split quasi-hereditary algebras}
\label{sec:highest_weight}

\subsection{Definitions and basic properties}
\label{ssec:hw_def_and _properties}

Originally, quasi-hereditary algebras were introduced in \cite{S}. Highest weight categories were introduced in \cite{CPS}, \cite{CPS2}. Rouquier \cite{Ro} generalized this to arbitrary noetherian commutative ring. In this subsection we follow Rouquier \cite{Ro} to define split quasi-hereditary algebras and highest weight categories over an arbitrary commutative ring $K,$ not necessarily noetherian.

From now on we denote by $\cP_K$ the category of finitely generated projective $K$-modules. We write $-\otimes-$ for $-\otimes_K-,$ and $\Hom(-,-)$ for $\Hom_K(-,-).$ We also write $(-)^*$ for $\Hom(-,K).$ All the modules over associative algebras are assumed to be right unless otherwise stated.

\begin{defi}A $K$-algebra (associative, unital) $A$ is called finite projective if $A\in\cP_K.$\end{defi}

From now on in this subsection we assume that $A$ is a finite projective $K$-algebra. We denote by $\cC=\Rep(K,A)$ the category of right $A$-modules which are finitely generated projective as $K$-modules. Clearly, $\cC$ is an exact category in which the exact structure is induced from the ambient abelian category $\text{Mod-}A.$

\begin{remark}Rouquier \cite{Ro} considers the case when $K$ is noetherian, and hence so is $A.$ He deals with the category of finitely generated right $A$-modules. However, it is easy to see that all the considerations can  be made in the exact category $\Rep(K,A),$ without any assumption on a commutative ring $K.$\end{remark}




\begin{defi}\label{def:heredity_ideal}(\cite{Ro}, Definition 4.1) An ideal $J\subseteq A$ is said to be indecomposable split heredity ideal if the following conditions hold:

(i) the $K$-module $A/J$ is projective;

(ii) $J$ is projective as a right $A$-module;

(iii) $J^2=J;$

(iv) the $K$-algebra $\End_A(J_A)$ is Morita equivalent to $K.$\end{defi}

\begin{defi}A full subcategory $\cM(\cC)\subset\cC$ consists of projective objects $L\in\cC$ which are faithful as $K$-modules, and satisfy the following condition:

- for any projective object $P\in\cC,$ the evaluation map $\tau_{L,P}':L\otimes\Hom(L,P)\to P$ is a split injection of $K$-modules.

We denote by $M(\cC)$ the set of isomorphism classes of objects in $\cM(\cC).$\end{defi}

Clearly, the set $M(\cC)$ is acted on by $\Pic(K).$

\begin{prop}\label{prop:bijection_for_M(C)/Pic(K)}(\cite{Ro}, Proposition 4.7) There is a bijection between $M(\cC)/Pic(K)$ and the set of indecomposable split heredity ideals $J\subset A,$ given by $L\mapsto \im(\tau_{L,A}').$\end{prop}

\begin{remark}1) Assume that $J\subset A$ is an indecomposable split heredity ideal. Put $B:=\End_A(J).$ Let $P$ be a right $B$-module which corresponds to $K$ via Morita equivalence $\text{Mod-}B\cong \text{Mod-}K.$ Put $$L:=P\otimes_B J.$$ Then $L\in\cM(\cC),$ and moreover its class in $M(\cC)/\Pic(K)$ corresponds to $J$ under the bijection from Proposition \ref{prop:bijection_for_M(C)/Pic(K)}.

2) Assume that $L\in\cM(\cC)$ is isomorphic to $eA$ for an idempotent $e\in A$ (this can always be achieved by replacing $A$ with $M_n(A)$). Then we have that $J=AeA\cong Ae\otimes eA,$ and $eAe=K.$\end{remark}

\begin{defi}Let $\Lambda$ be a poset. A subset $\Gamma\subseteq\Lambda$ is called an ideal (resp. a coideal) if for any $\gamma\in\Gamma$ the set $[-\infty,\gamma]:=\{\tau\in\Lambda\mid \tau\leq\gamma\}$ (resp. $[\gamma,+\infty]:=\{\tau\in\Lambda\mid \tau\geq\gamma\}$) is contained in $\Gamma.$
\end{defi}

Clearly, $\Gamma\subseteq\Lambda$ is an ideal if and only if $\Lambda\setminus\Gamma\subseteq\Lambda$ is a coideal.

\begin{defi}A split quasi-hereditary algebra over $K$ is a finite projective $K$-algebra $A,$ together with a finite poset $\Lambda$ and a set of (two-sided) ideals $\{\cI_{\Omega}\subseteq A\mid \Omega\subseteq\Lambda\text{ a coideal}\}$ satisfying the following conditions:

(i) if $\Omega\subseteq\Omega'\subseteq\Lambda$ are coideals, then $\cI_{\Omega}\subseteq \cI_{\Omega'};$

(ii) if $\Omega\subseteq\Omega'\subseteq\Lambda$ are coideals and $|\Omega'\setminus\Omega|=1,$ then $\cI_{\Omega'}/I_{\Omega}\subseteq A/\cI_{\Omega}$ is a split indecomposable heredity ideal;

(iii) $\cI_{\emptyset}=0$ and $\cI_{\Lambda}=A.$\end{defi}

\begin{prop}Let $K\to K'$ be a homomorphism of commutative rings.

1) If $A$ is a finite projective $K$-algebra, then $K'\otimes_K A$ is a finite projective $K'$-algebra. Moreover, if $J\subseteq A$ is a split heredity ideal, then $K'\otimes_K J\subseteq K'\otimes_K A$ is a split heredity ideal as well.

2) If $(A,\Lambda,\cI)$ is a split quasi-hereditary $K$-algebra, then $(K'\otimes_K A,\Lambda,K'\otimes_K\cI)$ is a split quasi-hereditary $K'$-algebra.\end{prop}

\begin{proof}This follows immediately from the definitions.\end{proof}

Let now $\cC$ be a $K$-linear exact category which is equivalent to $\Rep(K,A)$ for some finite projective $K$-algebra $A.$ For any finite set of objects $S\subset Ob(\cC),$ we put $\widetilde{S}:=\{D\otimes U\mid D\in S,\,U\in \cP_K\}.$ For a strictly full subcategory $\cD\subset\cC,$ we denote by $\cC^{\cD}\subseteq\cC$ the full subcategory of objects which have a finite filtration with subquotients in $\cD.$

Let $(\Delta\subseteq Ob(\cC),\leq)$ be a finite set of objects, together with a poset structure.

\begin{defi}\label{def:highest_weight_over_ring}A pair $(\cC,\Delta)$ as above is called a highest weight category if the following conditions hold:

(i) $\End_{\cC}(D)=K$ for each $D\in\Delta;$

(ii) For $D_1,D_2\in\Delta,$ if $\Hom_{\cC}(D_1,D_2)\ne 0$ then $D_1\leq D_2$ w.r.t. the poset structure on $\Delta;$

(iii) for any $D\in\Delta$ there exists a projective object $P\in\cC$ and a surjection $f:P\to D$ such that $\ker(f)\in\cC^{\widetilde{\Delta_{>D}}}.$

(iv) the subcategory $\cC^{\widetilde{\Delta}}\subset\cC$ contains a projective generator of $\cC$.\end{defi}

\begin{theo}(\cite{Ro}, Theorem 4.16)\label{th:quasi_her_equiv_highest_weight} Let $A$ be a finite projective $K$-algebra, and $\cC:=\Rep(K,A).$

Let $(A,\Lambda,\cI)$ be a split quasi-hereditary algebra. For any $\lambda\in\Lambda,$ take some object $\Delta(\lambda)\in \cM(\Rep(K,A/\cI_{\Omega_{>\lambda}}))\subset\Rep(K,A),$ such that its image in $M(\Rep(K,A/\cI_{\Omega_{>\lambda}}))/\Pic(K)$ corresponds to the ideal $\cI_{\Omega_{\geq\lambda}}/\cI_{\Omega_{>\lambda}}\subset A/\cI_{\Omega_{>\lambda}}$ under the bijection from Proposition \ref{prop:bijection_for_M(C)/Pic(K)}. Then $(\cC,\{\Delta(\lambda)\}_{\lambda\in\Lambda})$ is a highest weight category.

Conversely, suppose that $(\cC,\{\Delta(\lambda)\}_{\lambda\in\Lambda})$ is a highest weight category. Then for each coideal $\Omega\subseteq\Lambda$ define the ideal $\cI_{\Omega}\subseteq A$ to be the annihilator of all objects in $\cC^{\widetilde{\Delta\setminus\Omega}}.$ Then $(A,\Lambda,\cI)$ is a split quasi-hereditary algebra.\end{theo}

\begin{prop}\label{prop:costandard}(\cite{Ro}, Proposition 4.19) Let $(\cC,\{\Delta(\lambda)\}_{\lambda\in\Lambda})$ be a highest weight category. Then there is a unique (up to a natural isomorphism) collection of objects $\{\nabla(\lambda)\}_{\lambda\in\Lambda})$ in $\cC$ satisfying the following conditions:

(i) $(C^{op},\{\nabla(\lambda)\}_{\lambda\in\Lambda})$ is a highest weight category;

(ii) for $\lambda,\mu\in\Lambda$ we have $\Ext^i(\Delta(\lambda),\nabla(\mu))=\begin{cases}k & \text{if }\lambda=\mu, \,i=0\\
0 & \text{otherwise.}\end{cases}$\end{prop}

\begin{remark}In Proposition \ref{prop:costandard} (i), the order on $\{\nabla(\lambda)\}_{\lambda\in\Lambda})$ is the opposite of the order on $\{\Delta(\lambda)\}_{\lambda\in\Lambda}).$\end{remark}

\begin{prop}\label{prop:product_of_quasi_her}1) Let $(A_1,\Lambda_1,\cI_1)$ and $(A_2,\Lambda_2,\cI_2)$ be split quasi-hereditary $K$-algebras. Then $(A_1\otimes A_2,\Lambda_1\times\Lambda_2,\cI)$ is also a split quasi-hereditary algebra, where for a coideal $\Omega\subseteq\Lambda_1\times\Lambda_2$ we put
$$\cI_{\Omega}=\sum\limits_{(\lambda_1,\lambda_2)\in\Omega}(\cI_1)_{[\lambda_1,+\infty]}\otimes (\cI_2)_{[\lambda_2,+\infty]}.$$

2) Denote the collection of standard (resp. costandard) objects in $\Rep(K,A_i)$ by $\{\Delta_i(\lambda)\}_{\lambda\in\Lambda_i}$ (resp. $\nabla_i(\lambda)_{\lambda\in\Lambda_i}$). Then the collection of standard (resp. costandard) objects in $\Rep(K,A_1\otimes A_2)$ is exactly $\{\Delta_1(\lambda_1)\otimes\Delta_2(\lambda_2)\}_{(\lambda_1,\lambda_2)\in\Lambda_1\times\Lambda_2}$ (resp. $\{\nabla_1(\lambda_1)\otimes\nabla_2(\lambda_2)\}_{(\lambda_1,\lambda_2)\in\Lambda_1\times\Lambda_2}$)\end{prop}

\begin{proof}Indeed, by Theorem \ref{th:quasi_her_equiv_highest_weight} it suffices to show that $$(\Rep(K,A_1\otimes A_2),\{\Delta_1(\lambda_1)\otimes\Delta_2(\lambda_2)\}_{(\lambda_1,\lambda_2)\in\Lambda_1\times\Lambda_2})$$ is a highest weight category. This is proved by straightforward checking.\end{proof}

For an ideal $\Gamma\subseteq\Delta,$ denote by $\cC[\Gamma]\subseteq\cC$ the full subcategory which consists of objects $M\in Ob(\cC)$ such that $\Hom(D,M)=0$ for each $D\in\Delta\setminus\Gamma.$

\begin{prop}\label{prop:C[Gamma]}1) Let $(A,\Lambda,\cI)$ be a split quasi-hereditary algebra, and $\Omega\subset\Lambda$ a coideal. Then the algebra $A/I_{\Omega}$ has a natural structure of a split quasi-hereditary algebra $(A/I_{\Omega},\overline{\Lambda},\overline{\cI}).$ Here $\overline{\Lambda}=\Lambda\setminus\Omega,$ and for any coideal $\overline{\Omega}\subset \overline{\Lambda}$ we have $$\cI_{\overline{\Omega}}=\cI_{\overline{\Omega}\sqcup\Omega}/\cI_{\Omega}\subset A/\cI_{\Omega}.$$

2) In the assumptions of 1), if $(\cC=\Rep(K,A),\{\Delta(\lambda)\}_{\lambda\in\Lambda})$ is the corresponding highest weight category, then $(\cC[\overline{\Lambda}],\{\Delta_{\overline{\lambda}}\}_{\overline{\lambda}\in\overline{\Lambda}})$ is the highest weight category corresponding to $(A/I_{\Omega},\overline{\Lambda},\overline{\cI}).$ Moreover, the functor $D^b(\cC[\overline{\Lambda}])\to D^b(\cC)$ is fully faithful.\end{prop}

\begin{proof}2) is proved in \cite{Ro}, Proposition 4.13. By Theorem \ref{th:quasi_her_equiv_highest_weight}, 2) implies 1).\end{proof}

We now discuss the derived-categorical viewpoint on highest weight categories.

\begin{prop}\label{prop:D^b_of_highest_weight_category}1) Let $J\subseteq A$ be an indecomposable split heredity ideal in a finite projective $K$-algebra $A.$ Then we have a semi-orthogonal decomposition
$$\Perf(A)=\la \Perf(A/J), \Perf(K)\ra.$$

2) Suppose that $(A,\Lambda,\cI)$ is a split quasi-hereditary algebra. Then the algebra $A$ is homologically smooth and proper over $K.$ In particular, we have that $\Perf(A)=D^b(\Rep(K,A)).$

3) In the assumption of 2), let $\{\Delta(\lambda)\}_{\lambda\in\Lambda}$ be the poset of standard objects in $\Rep(K,A).$ Then $\{\Delta(\lambda)\}_{\lambda\in\Lambda}$ is a full exceptional collection in $D^b(\Rep(K,A)),$ and its left dual is exactly $\{\nabla(\lambda)\}_{\lambda\in\Lambda}.$\end{prop}

\begin{proof}1) Denote by $\pi:A\to A/J$ the projection homomorphism. Then the restriction of scalars functor $$\pi_*:D^b(\Rep(K,A/J))\to D^b(\Rep(K,A))$$ takes $\Perf(A/J)$ to $\Perf(A),$ since the object $A/J\in\Rep(K,A)$ has a projective resolution
 $$0\to J_A\to A\to A/J_A\to 0.$$ The right adjoint to $\pi_*$ is
 $$\pi^!:D^b(\Rep(K,A))\to D^b(\Rep(K,A/J)),\quad \pi^!(-)=\bR\Hom_A(A/J_A,-).$$
 By Definition \ref{def:heredity_ideal}, we have $\bR\Hom_A(J_A,A/J_A)=0.$ Hence, we have $\pi^!\pi_*(A/J)\cong A/J.$ Hence, the functor
 $$\pi_*:\Perf(A/J)\to\Perf(A)$$ is fully faithful. Let $L\in\cM(\Rep(K,A))$ be an object such that its class in $M(\Rep(K,A))/\Pic(K)$ corresponds to $J$ under the bijection from Proposition \ref{prop:bijection_for_M(C)/Pic(K)}. It defines a full embedding $-\otimes_K L:\Perf(K)\to \Perf(A).$ Clearly, we get a semi-orthogonal decomposition $$\Perf(A)=\la\pi_*(\Perf(A/J)),\Perf(K)\otimes_K L\ra.$$

2) Since $A$ is a finitely generated projective $K$-module, $A$ is proper over $K.$ We prove homological smoothness by induction on $|\Lambda|.$

If $|\Lambda|=0,$ then $A=0$ and there is nothing to prove.

Suppose that the statement is proved for $|\Lambda|<n,$ where $n\in\Z_{>0}.$ Let us prove it for $|\Lambda|=n.$
Take some maximal element $\lambda\in\Lambda.$ By 1), we have a semi-orthogonal decomposition
$$\Perf(A)=\la \Perf(A/I_{\{\lambda\}}),\Perf(K)\ra.$$
By Proposition \ref{prop:C[Gamma]}, the algebra $A/I_{\{\lambda\}}$ is split quasi-hereditary. Since $|\Lambda\setminus\{\lambda\}|=n-1,$ by induction hypothesis the algebra $A/I_{\{\lambda\}}$ is homologically smooth. Therefore, by Proposition \ref{prop:gluing_of_smooth_DG} the algebra $A$ is homologically smooth as well. This proves the statement of induction.

3) The fact that $\{\Delta(\lambda)\}_{\lambda\in\Lambda}$ is a full exceptional collection in $D^b(\Rep(K,A)),$ follows from 1) by induction, as in 2). The fact that $\{\nabla(\lambda)\}_{\lambda\in\Lambda}$ is the left dual exceptional collection, is a direct corollary of Proposition \ref{prop:costandard}, (ii).\end{proof}

\begin{lemma}\label{lem:criterion_for_cC^Delta}Let $A$ be a split quasi-hereditary $K$-algebra, with standard objects $\{\Delta(\lambda)\}_{\lambda\in\Lambda}$ and costandard objects $\{\nabla(\lambda)\}_{\lambda\in\Lambda}.$ Let $X\in\Perf(A)\simeq D^b(\Rep(K,A))$ be an object. Then the following are equivalent:

(i) $H^i(X)=0$ for $i\ne 0,$ and $H^0(X)\in\Rep(K,A)^{\widetilde{\Delta}};$

(ii) $\forall\lambda\in\Lambda,$ we have $\bR\Hom(X,\nabla(\lambda))\in \cP_K[0];$

(iii) $\forall\lambda\in\Lambda,$ we have $X\stackrel{\bL}{\otimes}_A\nabla(\lambda)^*\in \cP_K[0].$

Dually, the following are equivalent:

(i') $H^i(X)=0$ for $i\ne 0,$ and $H^0(X)\in\Rep(K,A)^{\widetilde{\nabla}};$

(ii') $\forall\lambda\in\Lambda,$ we have $\bR\Hom(\Delta(\lambda),X)\in \cP_K[0];$

(iii') $\forall\lambda\in\Lambda,$ we have $\Delta(\lambda)\stackrel{\bL}{\otimes}_A X^*\in \cP_K[0].$\end{lemma}

\begin{proof}Indeed, both equivalences $(i)\Leftrightarrow(ii)$ and $(i')\Leftrightarrow(ii')$ follow immediately from the fact that $\{\nabla(\lambda)\}_{\lambda\in\Lambda}$ is the left dual exceptional collection to $\{\Delta(\lambda)\}_{\lambda\in\Lambda}$ (by Proposition \ref{prop:D^b_of_highest_weight_category}). Equivalences $(ii)\Leftrightarrow(iii)$ and $(ii')\Leftrightarrow(iii')$ are implied by the natural isomorphisms
$$\bR\Hom(X,\nabla(\lambda))^*\cong X\stackrel{\bL}{\otimes}_A\nabla(\lambda)^*,\quad \bR\Hom(\Delta(\lambda),X)^*\cong \Delta(\lambda)\stackrel{\bL}{\otimes}_A X^*.$$
This proves the lemma.\end{proof}

\begin{defi}The objects of the category $\cC^{\widetilde{\Delta}}$ are called standardly filtered. The objects of the category $\cC^{\widetilde{\nabla}}$ are called costandardly filtered.\end{defi}

\begin{prop}\label{prop:C[Omega]}Let $(\cC,\Delta)$ be a highest weight category, and $\Omega\subset\Delta$ a coideal. For each $D\in\Omega,$ choose a projective object $P_{D}\in\cC$ with a surjection $f_{D}:P_{D}\to D$ such that $\ker(f_{D})\in \cC^{\widetilde{\Delta_{>D}}}.$

1) The algebra $A_{\Omega}:=\End_{\cC}(\bigoplus\limits_{D\in\Omega}P_{D})$ is split quasi-hereditary.

2) The subcategory $\add(\bigoplus\limits_{D\in\Omega}P_{D})\subset\cC$ does not depend on the choice of $P_{D}.$ In particular, the highest weight category $\cC[\Omega]:=\Rep(K,A_{\Omega})$ depends only on the coideal $\Omega\subset\Lambda.$

3) The functor $$\pi_{\Omega}:\cC\to\cC[\Omega],\quad \pi_{\Omega}(X)=\Hom_{\cC}(\bigoplus\limits_{D\in\Omega}P_{D},X),$$
is exact and induces a quotient functor $\bR\pi_{\Omega}:D^b(\cC)\to D^b(\cC[\Omega]),$ with kernel $D^b(\cC[\Delta\setminus\Omega]).$ The functor $\bR\pi_{\Omega}$ has a fully faithful left (resp. right) adjoint $i_{\Omega}$ (resp. $j_{\Omega}$), and we have semi-orthogonal decompositions
$$D^b(\cC)=\la\cC[\Delta\setminus\Omega],i_{\Omega}(\cC[\Omega])\ra,\quad D^b(\cC)=\la j_{\Omega}(\cC[\Omega]),\cC[\Delta\setminus\Omega]\ra.$$\end{prop}

\begin{proof}1) It suffices to show that the category $\Rep(K,A_{\Omega})$ is a highest weight category. For any $D\in\Omega,$ we have an object $\pi_{\Omega}(D)=\Hom_{\cC}(\bigoplus\limits_{E\in\Omega}P_{E},D)\in\Rep(K,A_{\Omega}).$ We claim that $(\Rep(K,A_{\Omega}),\pi_{\Omega}(\Omega))$ is a highest weight category.

First, note that the sets of objects $\Omega$ and $\{P_D\}_{D\in\Omega}$ generate the same triangulated subcategory of $D^b(\cC).$ It follows immediately that the properties (i) and (ii) from Definition \ref{def:highest_weight_over_ring} hold for $\pi_{\Omega}(\Omega).$ Further, we have that for any $D\in\Omega$ the $A_{\Omega}$-module $\pi_{\Omega}(P_D)$ is projective, and $\pi_{\Omega}(f_D):\pi_{\Omega}(P_D)\to \pi_{\Omega}(D)$ is a surjection with $\ker(\pi_{\Omega}(f_D))\in\Rep(K,A_{\Omega})^{\widetilde{\pi_{\Omega}(\Delta_{>D})}}.$ This verifies the property (iii) for $\pi_{\Omega}(\Omega).$ Finally, we have by definition that $\pi_{\Omega}(P_D)\in\Rep(K,A_{\Omega})^{\widetilde{\pi_{\Omega}(\Omega)}},$ hence $A_{\Omega}\cong\bigoplus\limits_{D\in\Omega}\pi_{\Omega}(P_D)\in\cC^{\widetilde{\pi_{\Omega}(\Omega)}}.$ This verifies the property (iv).

2) It suffices to note again that triangulated subcategory of $D^b(\cC)$ generated by $P_D,$ $D\in\Omega,$ coincides with triangulated subcategory generated by $\Omega,$ hence depends only on $\Omega,$ not on the choice of $P_D.$ Hence the subcategory $\add(\bigoplus\limits_{D\in\Omega}P_{D})\subset\cC$ also depends only on $\Omega.$

3) First we define the functor $i_{\Omega}:D^b(\cC[\Omega])\to D^b(\cC)$ by the formula $$i_{\Omega}(M):=M\stackrel{\bL}{\otimes}_{A_{\Omega}}(\bigoplus\limits_{D\in\Omega}P_{D}).$$
Clearly, it is left adjoint to $\bR\pi_{\Omega}.$ Moreover, we have that $\bR\pi_{\Omega}(i_{\Omega}(A_{\Omega}))\cong A_{\Omega},$ hence $\bR\pi_{\Omega}\circ i_{\Omega}(A_{\Omega})\cong\id.$ It follows that $\bR\pi_{\Omega}$ is a semi-orthogonal projection. Its kernel is the right orthogonal to $i_{\Omega}(D^b(\cC[\Omega])),$ which is by definition $D^b(\cC[\Delta\setminus\Omega])\subset D^b(\cC).$ The remaining decomposition $$D^b(\cC)=\la j_{\Omega}(\cC[\Omega]),\cC[\Delta\setminus\Omega]\ra$$ follows.\end{proof}

\begin{defi}1) Let $\Lambda$ be a poset. For $a,b\in\Lambda$ we put $$[a,b]:=\{c\in\Lambda\mid a\leq c\leq b\}.$$
A subset $\Theta\subset\Lambda$ is called convex if for any $a,b\in\Theta$ we have $[a,b]\in\Theta.$ Equivalently, $\Theta\subset\Lambda$ is convex if $\Theta=\Gamma\cap\Omega$ for some ideal $\Gamma\subseteq\Lambda$ and coideal $\Omega\subseteq\Lambda.$

2) Let $(\cC,\Delta)$ be a highest weight category. For a convex subset $\Theta\subseteq\Delta$ we define the highest weight category $\cC[\Theta]$ by the formula
$$\cC[\Theta]:=\cC[\Gamma][\Theta],$$
where $\Gamma\subseteq\Delta$ is an ideal generated by $\Theta$ (clearly, $\Theta$ is a coideal in $\Gamma,$ so $\cC[\Gamma][\Theta]$ is well-defined by Proposition \ref{prop:C[Omega]} 2)).\end{defi}

\subsection{Gluing of split quasi-hereditary algebras and highest weight categories}
\label{ssec:gluing_of_quasi_her}

As in the previous subsection, $K$ denotes the basic commutative ring.

Suppose that $A_1,\dots,A_m$ are finite projective $K$-algebras. Suppose that $$M_{ij}\in\Rep(K,A_i\otimes_K A_j^{op}),\quad 1\leq i<j\leq m,$$
is a collection of finitely generated $K$-projective bimodules. Also, suppose that the maps
$$m_{ijk}:M_{jk}\otimes_{A_j}M_{ij}\to M_{ik},\quad 1\leq i<j<k\leq m,$$
are given, satisfying the associativity condition
$$m_{ikl}\circ(\id_{M_{kl}}\otimes_{A_k}m_{ijk})=m_{ijl}\circ(m_{jkl}\otimes_{A_j}\id_{M_{ij}}).$$

\begin{defi}\label{def:gluing_of_A_i_via_M}The $K$-algebra $A_1\times_{M}A_2\dots\times_{M} A_m$ is defined by the formula
$$A_1\times_{M}A_2\dots\times_{M} A_m:=\bigoplus\limits_{i=1}^m A_i\oplus\bigoplus\limits_{1\leq i<j\leq m}M_{ij}.$$
The nonzero components of the multiplication map in $A_1\times_{M}A_2\dots\times_{M} A_m$ are:

- multiplication maps $A_i\otimes A_i\to A_i,$ $1\leq i\leq m;$

- bimodule structure maps $A_j\otimes M_{ij}\to M_{ij},$ $M_{ij}\otimes A_i\to M_{ij},$ $1\leq i<j\leq m;$

- the compositions $M_{jk}\otimes M_{ij}\to M_{jk}\otimes_{A_j} M_{ij}\stackrel{m_{ijk}}{\longrightarrow}M_{ik}.$\end{defi}

It is clear that $A_1\times_{M}A_2\dots\times_{M} A_m$ is a finite projective $K$-algebra. For convenience we denote it by $\widetilde{A}.$

Let us put $\cC_i:=\Rep(K,A_i),$ $1\leq i\leq m;$ $\widetilde{\cC}:=\Rep(K,\widetilde{A}).$ Further, we define the colimit-preserving functor
$$\phi_{ij}:\text{Mod-}A_j\to\text{Mod-}A_i,\quad \phi_{ij}(N)=N\otimes_{A_j}M_{ij},\quad 1\leq i<j\leq m.$$
It has a right adjoint given by the formula
$$\phi_{ij}^!:\text{Mod-}A_i\to\text{Mod-}A_j,\quad \phi_{ij}^!(N)=\Hom_{A_i}(M_{ij},N).$$
The maps $m_{ijk}$ induce the natural transformations $m_{ijk}:\phi_{ij}\circ\phi_{jk}\to \phi_{ik},$ which we denote by the same symbol. These natural transformations satisfy the analogous associativity conditions.

\begin{defi}\label{def:gluing_of_cC_i_via_phi}The (additive $K$-linear) category $\cC_1\times_{\phi}\cC_2\dots\times_{\phi}\cC_m$ is defined as follows.

An object of $\cC_1\times_{\phi}\cC_2\dots\times_{\phi}\cC_m$ is a collection $(\{N_i\},\{f_{ij}\})$ of objects $N_i\in\cC_i,$ $1\leq i\leq m,$ and morphisms $f_{ij}:\phi_{ij}(N_j)\to N_i$ in $\text{Mod-}A_i,$ $1\leq i<j\leq m,$ such that the following diagrams commute:
$$\begin{CD}
\phi_{ij}(\phi_{jk}(N_k)) @> m_{ijk}(N_k) >> \phi_{ik}(N_k)\\
@V \phi_{ij}(f_{jk}) VV                              @V f_{ik} VV\\
\phi_{ij}(N_j) @> f_{ij} >> N_i,
\end{CD}$$
where $1\leq i<j<k\leq m.$

A morphism $g:(\{N_i\},\{f_{ij}\})\to (\{N_i'\},\{f_{ij}'\})$ is a collection of morphisms $g_i:N_i\to N_i'$ in $\cC_i,$ $1\leq i\leq m,$ such that the following diagrams commute:
$$\begin{CD}
\phi_{ij}(N_j) @> \phi_{ij}(g_j) >> \phi_{ij}(N_j')\\
@V f_{ij} VV                              @V f_{ij}' VV\\
N_i @> g_i >> N_i',
\end{CD}$$
where $1\leq i<j\leq m.$ The composition is defined componentwise.\end{defi}

\begin{prop}The category $\widetilde{\cC}$ is naturally equivalent to the category $\cC_1\times_{\phi}\cC_2\dots\times_{\phi}\cC_m$ from Definition \ref{def:gluing_of_cC_i_via_phi}.
\end{prop}

\begin{proof}Indeed, let us define the functor $F:\cC_1\times_{\phi}\cC_2\dots\times_{\phi}\cC_m\to \widetilde{\cC}.$ As $K$-modules,
$$F(\{N_i\},\{f_{ij}\}):=\bigoplus\limits_{i=1}^m N_i.$$ The multiplication map $N_i\otimes A_j\to F(\{N_i\},\{f_{ij}\})$ is zero for $i\ne j,$ and is given by the $A_j$-module structure on $N_j$ for $i=j.$ Finally, the multiplication map $N_i\otimes M_{jk}\to F(\{N_i\},\{f_{ij}\})$ is zero for $i\ne k,$ and is given by the composition
$$N_k\otimes M_{jk}\to N_k\otimes_{A_k} M_{jk}\stackrel{f_{jk}}{\to} N_j$$
for $i=k.$

A straightforward checking shows that $F$ is a $K$-linear equivalence of categories.\end{proof}

Let us define the functor $G_i^*:\text{Mod-}A_i\to \text{Mod-}\widetilde{A}$ by the formula
$$G_i^*(N):=N\oplus\bigoplus\limits_{j=1}^{i-1}\phi_{ji}(N),$$
with the obvious $\widetilde{A}$-module structure. Clearly, $G_i^*$ is colimit-preserving. Its right adjoint is given by the formula
$$G_{i*}:\text{Mod-}\widetilde{A}\to \text{Mod-}A_i,\quad G_{i*}(\widetilde{N}):==\widetilde{N}e_i,$$
where $e_i\in\widetilde{A}$ is the idempotent given by the identity $1\in A_i.$ The functor $G_{i*}$ also has a right adjoint, given by the formula
$$G_i^!:\text{Mod-}A_i\to \text{Mod-}\widetilde{A},\quad G_i^!(N):=N\oplus\bigoplus\limits_{j=i+1}^{m}\phi_{ij}^!(N),$$ with the obvious $\widetilde{A}$-module structure.

We also define the functor $F_i:\text{Mod-}A_i\to \text{Mod-}\widetilde{A}$ by the formula
$F_i(N):=N.$ The right $\widetilde{A}$-module structure is obvious: the multiplication map $N\otimes A_j\to N$ is zero for $j\ne i,$ and coincides with $A_i$-module structure map for $j=i;$ the multiplication map $N\otimes M_{ij}\to N$ is zero.

We need the following auxiliary result.

\begin{lemma}\label{lem:equivalence_of_5_conditions_on_bimodule}Let $B$ and $B'$ be split quasi-hereditary $K$-algebras. Denote by $\{\Delta(\lambda)\}_{\lambda\in\Lambda}$ (resp. $\{\Delta'(\lambda')\}_{\lambda'\in\Lambda'}$) the poset of standard objects in $\Rep(K,B)$ (resp. in $\Rep(K,B')$). Similarly, denote by $\{\nabla(\lambda)\}_{\lambda\in\Lambda}$ (resp. $\{\nabla'(\lambda')\}_{\lambda'\in\Lambda'}$) the poset of costandard objects in $\Rep(K,B)$ (resp. in $\Rep(K,B')$). Let $M\in\Rep(K,B^{op}\otimes B')$ Then the following are equivalent:

(a) $\forall\lambda\in\Lambda,$ $\Delta(\lambda)\stackrel{\bL}{\otimes}_B M=\Delta(\lambda)\otimes_B M\in\Rep(K,B')^{\widetilde{\Delta'}};$

(b) $\forall\lambda'\in\Lambda',$ $\bR\Hom_{B'}(M,\nabla'(\lambda'))=\Hom_{B'}(M,\nabla'(\lambda'))\in \Rep(K,B)^{\widetilde{\nabla}};$

(c) $\forall\lambda'\in\Lambda',$ $M\stackrel{\bL}{\otimes}_{B'}\nabla'(\lambda')^*=M\otimes_{B'}\nabla'(\lambda')^*
\in\Rep(K,B)^{\widetilde{\nabla^*}};$

(d) $\forall\lambda\in\Lambda,$ $\bR\Hom_{B}(M,\Delta(\lambda)^*)=\Hom_{B}(M,\Delta(\lambda)^*)\in \Rep(K,B')^{\widetilde{(\Delta')^*}}.$

(e) $M$ is standardly filtered.\end{lemma}

\begin{proof}We first note that by Lemma \ref{lem:criterion_for_cC^Delta}, $(a)$ is equivalent to the following:
\begin{equation}\label{criterion_for_(a)}\bR\Hom_{B'}(\Delta(\lambda)\stackrel{\bL}{\otimes}_B M,\nabla'(\lambda'))\in\cP_K[0],\quad \lambda\in\Lambda,\,\lambda'\in\Lambda'.\end{equation}

$(a)\Leftrightarrow(b).$ By Lemma \ref{lem:criterion_for_cC^Delta}, $(b)$ is equivalent to
$$\bR\Hom_{B}(\Delta(\lambda),\bR\Hom_{B'}(M,\nabla'(\lambda')))\in\cP_K[0],\quad \lambda\in\Lambda,\,\lambda'\in\Lambda',$$
which is in turn equivalent to \eqref{criterion_for_(a)} by adjunction. This proves the equivalence $(a)\Leftrightarrow(b).$

$(e)\Rightarrow(a).$ Let us note that $$\Delta(\lambda_1)\stackrel{\bL}{\otimes}_B \nabla(\lambda_2)^*\cong\bR\Hom_{B}(\Delta(\lambda_1),\nabla(\lambda_2))^*=\begin{cases}K & \text{if }\lambda_1=\lambda_2;\\ 0 & \text{otherwise.}\end{cases}.$$
It follows immediately that $$\Delta(\lambda_1)\stackrel{\bL}{\otimes}_B(\nabla(\lambda_2)^*\otimes\Delta'(\lambda'))\in \Rep(K,B')^{\widetilde{\Delta'}}$$ for any $\lambda_1,\lambda_2\in\Lambda,$ $\lambda'\in\Lambda'.$ But by Proposition \ref{prop:product_of_quasi_her} standard objects in $\Rep(K,B^{op}\otimes B')$ are exactly of the form $\nabla(\lambda)^*\otimes\Delta'(\lambda'),$ $\lambda\in\Lambda,$ $\lambda'\in\Lambda'.$ This proves $(e)\Rightarrow(a).$

$(a)\Rightarrow(e).$ By Lemma \ref{lem:criterion_for_cC^Delta}, it suffices to check that
$$M\stackrel{\bL}{\otimes}_{B^{op}\otimes B'}(\Delta(\lambda)\otimes\nabla'(\lambda')^*)\in\cP_K[0],\quad\lambda\in\Lambda,\,\lambda'\in\Lambda'.$$
But we have
$$M\stackrel{\bL}{\otimes}_{B^{op}\otimes B'}(\Delta(\lambda)\otimes\nabla'(\lambda')^*)\cong (\Delta(\lambda)\stackrel{\bL}{\otimes}_B M)\stackrel{\bL}{\otimes}_{B'}\nabla(\lambda')^*\in\cP_K[0],$$
since $\Delta(\lambda)\stackrel{\bL}{\otimes}_B M=\Delta(\lambda)\otimes_B M$ is standardly filtered. This proves $(a)\Rightarrow(e).$

The equivalences $(c)\Leftrightarrow(d)$ and $(c)\Leftrightarrow(e)$ are obtained by replacing the pair $(A_1,A_2)$ with $(A_2^{op},A_1^{op}).$ Lemma is proved.
\end{proof}

\begin{theo}\label{th:gluing_of_quasi_hereditary_alg}Suppose that the algebras $A_1,\dots,A_m$ are split quasi-hereditary, so that $(\cC_i,\Delta_i=\{\Delta_i(\lambda)\}_{\lambda\in\Lambda_i}),$ $1\leq i\leq m,$ are highest weight categories. Assume that the following condition holds:

$(\star)$ the object $M_{ij}\in\Rep(K,A_i\otimes A_j^{op})$ is standardly filtered for $1\leq i<j\leq k.$

Then the algebra $\widetilde{A}$ (resp. the category $\widetilde{\cC}$) has two natural structures of a split quasi-hereditary algebra (resp. of a highest weight category).

1) In the first structure, the set of standard objects $\Delta_{(1)}\subset \widetilde{\cC}$ is exactly
 $$\{\Delta_{(1)}(\lambda_i):=G_i^*(\Delta_i(\Lambda_i))\mid \lambda_i\in\Lambda_i,\,1\leq i\leq m\}.$$
 The set of costandard objects $\nabla_{(1)}\subset\widetilde{\cC}$ is exactly
 $$\{\nabla_{(1)}(\lambda_i):=F_i(\nabla_i(\lambda_i))\mid \lambda_i\in\Lambda_i,\,1\leq i\leq m\}.$$
 The poset structure on $\Delta_{(1)}$ is the following: for $\lambda_i\in\Lambda_i,$ $\lambda_j'\in\Lambda_j$ we have
 $$\Delta_{(1)}(\lambda_i)< \Delta_{(1)}(\lambda_j')\Leftrightarrow \begin{cases}i<j\\
 i=j,\quad \lambda_i<\lambda_j'.\end{cases}$$

 2) In the second structure, the set of standard objects $\Delta_{(2)}\subset \widetilde{\cC}$ is exactly
 $$\{\Delta_{(2)}(\lambda_i):=F_i(\Delta_i(\Lambda_i))\mid \lambda_i\in\Lambda_i,\,1\leq i\leq m\}.$$
 The set of costandard objects $\nabla_{(2)}\subset\widetilde{\cC}$ is exactly
 $$\{\nabla_{(2)}(\lambda_i):=G_i^!(\nabla_i(\lambda_i))\mid \lambda_i\in\Lambda_i,\,1\leq i\leq m\}.$$
 The poset structure on $\Delta_{(2)}$ is the following: for $\lambda_i\in\Lambda_i,$ $\lambda_j'\in\Lambda_j$ we have
 $$\Delta_{(2)}(\lambda_i)< \Delta_{(2)}(\lambda_j')\Leftrightarrow \begin{cases}i>j\\
 i=j,\quad \lambda_i<\lambda_j'.\end{cases}$$\end{theo}

 \begin{proof}By Theorem \ref{th:quasi_her_equiv_highest_weight}, it suffices to check that $\cC$ has indeed the two structures of a highest weight category with required properties.

 By Lemma \ref{lem:equivalence_of_5_conditions_on_bimodule}, condition $(\star)$ implies that the functors $\phi_{ij}$ induce exact functors
 \begin{equation}\label{functors_on_standardly_filtered}C_j^{\widetilde{\Delta_j}}\to C_i^{\widetilde{\Delta_i}},\quad 1\leq i<j\leq m.\end{equation}

 1) We now verify the conditions of Definition \ref{def:highest_weight_over_ring} for $\Delta_{(1)}$. By adjunction, for $\lambda_i\in\Lambda_i$ we have
 $$\End_{\widetilde{\cC}}(\Delta_{(1)}(\lambda_i))\cong\Hom_{\cC_i}(\Delta_i(\lambda_i),G_{i*}G_i^*\Delta_i(\lambda_i))=\End_{\cC_i}(\Delta_i(\lambda_i))=K.$$
 This verifies (i).

 Suppose that for some $\lambda_i\in\Lambda_i,$ $\lambda_j'\in\Lambda_j,$ we have
 $$\Hom_{\widetilde{\cC}}(\Delta_{(1)}(\lambda_i),\Delta_{(1)}(\lambda_j'))\ne 0.$$
 Then $G_{i*}(G_j^*(\Delta_j(\lambda_j')))\ne 0,$ hence $i\leq j.$ Suppose that $i=j.$ Then
 $$\Hom_{\widetilde{\cC}}(\Delta_{(1)}(\lambda_i),\Delta_{(1)}(\lambda_j'))\cong
 \Hom_{\cC_i}(\Delta_i(\lambda_i),G_{i*}G_i^*\Delta_i(\lambda_i'))=\Hom_{\cC_i}(\Delta_i(\lambda_i),\Delta_i(\lambda_i'))\ne 0,$$
 hence $\lambda_i\leq \lambda_i'.$ This verifies (ii).

Take some $\lambda_i\in\Lambda_i,$ and a surjective morphism $f:P_i\to\Delta_i(\lambda_i),$ where $P_i\in\cC_i$ is projective, and $\ker(f)\in\cC_i^{\widetilde{\Delta_i(>\lambda_i)}}.$ Then put $\widetilde{P_i}:=G_i^*(P_i),$ and $\widetilde{f}:=G_i^*(f):\widetilde{P_i}\to\Delta_{(1)}(\lambda_i).$
Clearly, $\widetilde{f}$ is surjective, and exactness of \eqref{functors_on_standardly_filtered} implies that $$\ker(\widetilde{f})=G_i^*(\ker(f))\in\widetilde{\cC}^{\widetilde{\Delta_{(1)}(>\lambda_i)}}.$$ This verifies (iii).

Finally, let $P_i\in\cC_i$ be a projective generator which is contained in $\cC_i^{\widetilde{\Delta_i}}$ (e.g. $P_i=A_i$). Then $\bigoplus\limits_{i=1}^m G_i^*(P_i)\in\widetilde{\cC}$ is a projective generator which is contained in $\widetilde{\cC}^{\widetilde{\Delta_{(1)}}}.$ This verifies (iv). Therefore, $(\widetilde{\cC},\Delta_{(1)})$ is a highest weight category.

Now we identify the set of costandard objects $\nabla_{(1)}\subset\widetilde{\cC}.$ Exactness of \eqref{functors_on_standardly_filtered} implies that $\bL_{>0} G_i^*(\Delta_i(\lambda_i))=0.$ Hence, by adjunction
\begin{multline*}\bR\Hom_{D^b(\widetilde{\cC})}(\Delta_{(1)}(\lambda_i),\nabla_{(1)}(\lambda_j'))=
\bR\Hom_{D^b(\cC_i)}(\Delta_i(\lambda_i),G_{i*}F_j\nabla_j(\lambda_j'))\\=
\begin{cases}\bR\Hom_{D^b(\cC_i)}(\Delta_i(\lambda_i),\nabla_i(\lambda_i')) & \text{if }i=j;\\
0 & \text{otherwise.}\end{cases}\end{multline*}
Since $(\cC_i,\{\Delta_i(\lambda_i)\}_{\lambda_i\in\Lambda_i})$ is a highest weight category, for any $\lambda_i\in\Lambda_i,$ $\lambda_j'\in\Lambda_j$ we get
$$\bR\Hom_{D^b(\widetilde{\cC})}(\Delta_{(1)}(\lambda_i),\nabla_{(1)}(\lambda_j'))=\begin{cases}K & \text{for }i=j,\,\lambda_i=\lambda_i';\\
0 & \text{otherwise.}\end{cases}$$
This implies that $\nabla_{(1)}\subset\widetilde{\cC}$ is indeed the set of costandard objects.

2) For the second structure, we will show that it is obtained by duality from the first one. Namely, let us note that $$\widetilde{A}^{op}\cong A_m^{op}\times_{M}A_{m-1}^{op}\dots\times_{M} A_1^{op}.$$ It follows from 1) that the first highest weight structure on $\Rep(K,\widetilde{A}^{op})$ exists. It is easy to see that the standard (resp. costandard) objects in this structure are exactly $\nabla_{(2)}(\lambda_i)^*$ (resp. $\Delta_{(2)}(\lambda_i)^*$), where $\lambda_i\in\Lambda_i,$ $1\leq i\leq m.$ This proves existence of the second highest weight structure on $\widetilde{\cC}.$ Theorem is proved.
\end{proof}

\begin{example} Let us assume for simplicity that $K$ is a field. In the above notation, suppose that $A_i=K$ for $1\leq i\leq m,$ so $M_{ij}$ are finite-dimensional vector spaces and $\widetilde{A}$ is a path algebra of some directed quiver with relations. The condition $(\star)$ of Theorem \ref{th:gluing_of_quasi_hereditary_alg} is automatically satisfied.

In the first highest weight structure on $\Rep(K,\widetilde{A}),$ the standard objects are indecomposable projectives $P_i,$ $1\leq i\leq m,$ and costandard objects are simples $S_i,$ $1\leq i\leq m.$

In the second highest weight structure on $\Rep(K,\widetilde{A}),$ the standard objects are simples $S_i,$ $1\leq i\leq m,$ and costandard objects are indecomposable injectives $I_i,$ $1\leq i\leq m.$\end{example}

\section{Strict polynomial functors, representations of $GL_n$ and Koszul duality}
\label{sec:pol_functors_GL_n_Koszul}

 Again, we fix some basic commutative ring $K.$

\subsection{Strict polynomial functors}
\label{ssec:pol_functors}

Strict polynomial functors were introduced and studied by Friedlander and Suslin \cite{FS}. In this subsection we mostly follow Krause \cite{Kr}.

As above we denote by $\P_K$ the $K$-linear additive category of finitely generated projective $K$-modules.

Denote by $\mS_d$ the symmetric group on $\{1,\dots,d\}.$

\begin{defi}Let $V$ be a finitely generated projective $K$-module. The $d$-th divided power of $V$ is by definition $$\Gamma^dV:=(V^{\otimes d})^{\mS_d}\in\P_K.$$
The $d$-th symmetric power of $V$ is by definition
$$\Sym^d V:=(V^{\otimes d})_{\mS_d}\in\P_K.$$\end{defi}

Clearly, we have a functorial isomorphism $\Gamma^d(V^*)\cong (\Sym^dV)^*.$ To check that $\Gamma^dV$ and $\Sym^dV$ are indeed in $\P_K,$ it suffices to note that in the case when $V$ is free finitely generated, both $\Gamma^dV$ and $\Sym^dV$ are free finitely generated as well.

\begin{prop}\label{prop:gamma^d_monoidal} For $V,W\in\P_K$ we have a natural morphism $\Gamma^dV\otimes\Gamma^dW\to\Gamma^d(V\otimes W)$ (restriction of the isomorphism $V^{\otimes d}\otimes W^{\otimes d}\stackrel{\sim}{\to}(V\otimes W)^{\otimes d}$). Also, we have a natural isomorphism $\Gamma^dK\cong K.$ These morphisms endow $\Gamma^d$ with a structure of a symmetric monoidal functor.\end{prop}

\begin{proof}Straightforward checking.\end{proof}

Define the category $\Gamma^d\P_K$ as follows. Objects of $\Gamma^d\P_K$ are the same as of $\P_K.$ Further, we define
$$\Hom_{\Gamma^d\P_K}(V,W):=\Gamma^d\Hom_{\P_K}(V,W).$$
The composition is induced by the symmetric monoidal structure on $\Gamma^d$ from Proposition \ref{prop:gamma^d_monoidal}.

\begin{defi}The category of strict polynomial functors of degree $d$ over $K$ is the category
$$\Pol^d_K:=\Gamma^d\P_K\text{-Mod}=\Fun_K(\Gamma^d\P_K,K\text{-Mod}),$$
where in the RHS we take $K$-linear additive functors.

The strict polynomial functor $F$ is called finite if $F(V)\in P_K$ for any $V\in\P_K.$ The full subcategory of finite strict polynomial functors is denoted by $\pol^d_K\subseteq\Pol^d_K.$\end{defi}

The category $\Pol^d_K$ is abelian with infinite exact direct sums. Further, $\pol^d_K\subset\Pol^d_K$ is an exact subcategory. For any $V\in Ob(\P_K)=Ob(\Gamma^d\P_K),$ we denote by $\Gamma^{d,V}\in \pol^d_K$ the functor which is corepresented by $V.$ In other words,
$$\Gamma^{d,V}(W)=\Gamma^d\Hom(V,W).$$
By Yoneda Lemma, for any $F\in\Pol^d_K$ we have $$\Hom_{\Pol^d_K}(\Gamma^{d,V},F)=F(V),$$ hence $\Gamma^{d,V}$ is a projective object of $\Pol^d_K$ and $\pol^d_K.$

The bifunctors $-\otimes-$ and $\Hom(-,-)$ for $\P_K$ induce similar bifunctors for $\Gamma^d\P_K:$
$$-\otimes_K-:\Gamma^d\P_K\times\Gamma^d\P_K\to\Gamma^d\P_K,$$
$$\Hom_K(-,-):(\Gamma^d\P_K)^{op}\times\Gamma^d\P_K\to\Gamma^d\P_K,$$
which act in the same way on objects.

\begin{defi}1) The internal tensor product bifunctor
$$-\otimes_{\Gamma^d_K}-:\Pol^d_K\times\Pol^d_K\to\Pol^d_K$$
is a unique bifunctor which commutes with small colimits in both arguments and extends the bifunctor $-\otimes-$ on $\Gamma^d\P_K$ via Yoneda embedding. That is, we have $$\Gamma^{d,V}\otimes_{\Gamma^d_K}\Gamma^{d,W}=\Gamma^{d,V\otimes W}$$
for $V,W\in\P_K.$

2) The internal Hom functor
$$\cHom_{\Gamma^d_K}(-,-):(\Pol^d_K)^{op}\times\Pol^d_K\to\Pol^d_K$$
is just an internal Hom for the symmetric monoidal structure given by the tensor product $-\otimes_{\Gamma^d_K}-$ defined in 1). It is a unique bifunctor such that
$$\cHom_{\Gamma^d_K}(\Gamma^{d,V},F)=F\circ (V\otimes-)$$
and for each $F\in\Pol^d_K$ the functor $\cHom_{\Gamma^d_K}(-,F)$ takes colimits to limits.
\end{defi}

\begin{defi}The duality functor $(-)^\circ:(\Pol^d_K)^{op}\to\Pol^d_K$ is defined by the formula
$$F^{\circ}(V):=(F(V^*))^*.$$\end{defi}

Clearly, the functor $(-)^{\circ}$ preserves the subcategory $\pol^d_K.$ Further, we have a natural transformation
$(-)\to (-)^{\circ\circ},$ which is an isomorphism on $\pol^d_K.$ In particular, the functor
$$(-)^{\circ}:(\pol^d_K)^{op}\to \pol^d_K$$ is an anti-involution.

We define external tensor product
$$-\boxtimes-:\Pol^{d_1}_K\times\Pol^{d_2}_K\to\Pol^{d_1+d_2}_K$$
by the formula
$$F\boxtimes G(-):=F(-)\otimes G(-).$$
The other way to define the bifunctor $-\boxtimes-$ is to take a natural functor
$$F^{d_1,d_2}:\Gamma^{d_1+d_2}\P_K\to \Gamma^{d_1}\P_K\otimes \Gamma^{d_2}\P_K,$$
defined on objects by the formula $F^{d_1,d_2}(V):=(V,V),$ and on morphisms as a natural inclusion
\begin{multline*}\Gamma^{d_1+d_2}\Hom(V,W)\stackrel{\sim}{\to} (\Hom(V,W)^{\otimes (d_1+d_2)})^{\mS_{d_1+d_2}}\to (\Hom(V,W)^{\otimes (d_1+d_2)})^{\mS_{d_1}\times\mS_{d_2}}\\\stackrel{\sim}{\to}\Gamma^{d_1}\Hom(V,W)\otimes\Gamma^{d_2}\Hom(V,W),\quad V,W\in\P_K.\end{multline*}
Then the bifunctor $-\boxtimes-$ is a natural composition
$$\Gamma^{d_1}\P_K\text{-Mod}\times\Gamma^{d_2}\P_K\text{-Mod}\to (\Gamma^{d_1}\P_K\otimes \Gamma^{d_2}\P_K)\text{-Mod}\to \Gamma^{d_1+d_2}\P_K\text{-Mod},$$
where the first arrow is obvious and the second arrow is the composition with $F^{d_1,d_2}.$

\subsection{Schur algebras.}
\label{ssec:Schur_algebras}

\begin{defi}(\cite{Gr}, Theorem 2.6c) For any non-negative integers $n,d\in\Z_{\geq 0},$ Schur algebra $S_K(n,d)$ is defined by the formula
$$S_K(n,d):=\Gamma^d\Mat_{n\times n}(K).$$\end{defi}

\begin{theo}\label{th:modules_over_Schur}(\cite{BFS}, Theorem 2.4) The evaluation at $K^n$ defines a functor $\Pol^d_K\to S_K(n,d)\text{-Mod},$ which is an equivalence for $n\geq d.$ In particular, $\pol_K^d\simeq \Rep(K,S_K(n,d)^{op})$ for $n\geq d.$\end{theo}

The proof of Theorem \ref{th:modules_over_Schur} is quite easy. Indeed, we have a decomposition
$$\Gamma^d(V_1\oplus V_2)\cong\bigoplus\limits_{p=0}^d\Gamma^pV_1\otimes\Gamma^{d-p}V_2,\quad V_1,V_2\in\P_K.$$
It follows that for any $n\in\Z_{\geq 0}$ we have that
$$\Gamma^{d,K^n}\cong\bigoplus\limits_{\lambda\in\Lambda(n,d)}\Gamma^{\lambda},$$
where $$\Lambda(n,d)=\{(\lambda_1,\dots,\lambda_n)\in\Z_{\geq 0}^n\mid \lambda_1+\dots+\lambda_n=d\},$$
and $$\Gamma^{\lambda}:=\Gamma^{\lambda_1}\boxtimes\dots\boxtimes\Gamma^{\lambda_n}.$$
Clearly, (the isomorphism class of) $\Gamma^{\lambda}$ does not depend on the order of $(\lambda_1,\dots,\lambda_n),$ hence we may reorder to obtain a Young diagram, i.e. assume that $\lambda_1\geq\dots\geq\lambda_n.$

It follows that for each Young diagram $\lambda$ with $|\lambda|=d$ the object $\Gamma^{\lambda}$ is a direct summand of $\Gamma^{d,K^n}$ for $n\geq d.$ Hence, for any $V\in\P_K,$ the object $\Gamma^{d,V}\in\Pol^d_K$ is a direct summand of a direct sum of copies of $\Gamma^{d,K^n}.$ Therefore, $\Gamma^{d,K^n}$ is a progenerator of $\Pol^d_K$ for $d\geq 0.$ This implies Theorem \ref{th:modules_over_Schur}.

The following was proved in \cite{CPS2} in a more general case ($q$-Schur algebras).

\begin{theo}(\cite{CPS2}, Theorem 3.7.2) The $K$-algebra $S_K(n,d)$ is split quasi-hereditary for any $n,d\geq 0.$ In particular, $\pol_K^d$ is a highest weight category.\end{theo}

We now recall standard and costandard objects in $\pol_K^d.$

For a finite sequence of non-negative integers $\lambda=(\lambda_1,\dots,\lambda_n)\in\Lambda(n,d),$ put
$$\Lambda^{\lambda}:=\Lambda^{\lambda_1}\boxtimes\dots\boxtimes \Lambda^{\lambda_n}\in\pol^d_K,\quad \Sym^{\lambda}:=\Sym^{\lambda_1}\boxtimes\dots\boxtimes \Sym^{\lambda_n}\in\pol^d_K.$$

\begin{defi}\label{def:Schur_and_Weyl} For any Young diagram $\lambda$ of weight $d,$ define $\sigma_{\lambda}\in\mS_d$ to be a permutation defined by the formula $$\sigma_{\lambda}(\lambda_1+\dots+\lambda_{i-1}+j)=\lambda_1'+\dots+\lambda_{j-1}'+i,\quad i\geq 1,\,1\leq j\leq\lambda_i.$$
One defines a Schur functor $S_{\lambda}\in\pol^d_K$ and a Weyl functor $W_{\lambda}\in\pol^d_K$ by the formulas
$$S_{\lambda}(V)=\im(\Lambda^{\lambda'}\to V^{\otimes d}\stackrel{s_{\lambda}}{\to}V^{\otimes d}\to \Sym^{\lambda}(V)),$$
$$W_{\lambda}(V)=\im(\Gamma^{\lambda}\to V^{\otimes d}\stackrel{s_{\lambda'}}{\to}V^{\otimes d}\to \Lambda^{\lambda'}(V)).$$
Here for any Young diagram $\mu$ of weight $d$ we define
$$s_{\mu}:V^{\otimes d}\to V^{\otimes d},\quad s_{\mu}(v_1\otimes\cdots\otimes v_d)=v_{\sigma_{\mu}(1)}\otimes\cdots\otimes v_{\sigma_{\mu}(d)}.$$
The map $\Lambda^{\lambda'}(V)\to V^{\otimes d}$ (resp. $\Gamma^{\lambda'}(V)\to V^{\otimes d}$) is the tensor product of inclusions $\Lambda^{\lambda_j'}(V)\to V^{\otimes\lambda_j'}$ (resp. $\Gamma^{\lambda_i}(V)\to V^{\otimes\lambda_i}$). The map $V^{\otimes d}\to \Sym^{\lambda}(V)$ (resp. $V^{\otimes d}\to \Lambda^{\lambda'}(V)$) is the tensor product of projections $V^{\otimes \lambda_i}\to \Sym^{\lambda_i}(V)$ (resp. $V^{\otimes \lambda_j'}\to \Lambda^{\lambda_j'}(V)$).\end{defi}

Weyl functors $W_{\lambda}$ are the standard objects of $\pol_K^d.$ The partial order is the dominance:
$\lambda\unlhd\mu$ if $$\lambda_1+\dots\lambda_i\leq\mu_1+\dots\mu_i\text{ for }i\geq 0.$$

Schur functors $S_{\lambda}$ are the costandard objects. In particular, we have that
$$\bR\Hom(W_{\lambda},S_{\mu})=\begin{cases}K & \text{for }\lambda=\mu;\\
0 & \text{otherwise.}\end{cases}$$

We will need the following classical result, which is the universal form of Littlewood-Richardson rule.

\begin{theo}\label{th:Littlewood-Richardson}Take some Young diagrams $\lambda$ and $\mu$ with $|\lambda|=d_1,$ $|\mu|=d_2.$

1) The external  tensor product $S_{\lambda}\boxtimes S_{\mu}\in\pol^{d_1+d_2}_K$ has a finite filtration, with subquotients isomorphic to $S_{\nu}^{\oplus c_{\lambda,\mu}^{\nu}},$ where $c_{\lambda,\mu}^{\nu}$ are Littlewood-Richardson coefficients.

2) Similarly, the external tensor product $W_{\lambda}\boxtimes W_{\mu}\in\pol^{d_1+d_2}_K$ has a finite filtration, with subquotients isomorphic to $W_{\nu}^{\oplus c_{\lambda,\mu}^{\nu}}.$\end{theo}

\begin{proof}Part 1) is proved in \cite{Bo}, Theorem 3.7. Part 2) is obtained from 1) by applying the functor $(-)^{\circ}.$\end{proof}

\subsection{Koszul duality}
\label{ssec:Koszul_duality}

The bifunctor $-\otimes_{\Gamma^d_K}-$ is not biexact. One may replace it with a derived bifunctor

$$-\stackrel{\bL}{\otimes}_{\Gamma^d_K}-:D(\Pol^d_K)\times D(\Pol^d_K)\to D(\Pol^d_K).$$
Similarly, one has the derived internal Hom functor:
$$\bR\cHom_{\Gamma^d_K}:D(\Pol^d_K)^{op}\times D(\Pol^d_K)\to D(\Pol^d_K).$$

The following was shown in \cite{Kr}.

\begin{theo}(\cite{Kr}, Theorem 4.9) The functors $\Lambda^d\stackrel{\bL}{\otimes}_{\Gamma^d_K}-$ and $\bR\cHom_{\Gamma^d_K}(\Lambda^d,-)$ provide mutually inverse equivalences
$$D(\Pol^d_K)\simeq D(\Pol^d_K),\quad D^b(\pol^d_K)\simeq D^b(\pol^d_K).$$\end{theo}

\begin{defi}\label{def:koszul}We call $\bR\cHom_{\Gamma^d_K}(\Lambda^d,-)$ the "Koszul duality functor", and $\Lambda^d\stackrel{\bL}{\otimes}_{\Gamma^d_K}-$ the "inverse Koszul duality functor".\end{defi}

\begin{prop}\label{prop:properties_of_koszul}(\cite{Kr}, Propositions 4.8, 4.16) For any Young diagram $\lambda$ of weight $d,$ we have natural isomorphisms:
$$\Lambda^d\stackrel{\bL}{\otimes}_{\Gamma^d_K}\Gamma^{\lambda}\cong\Lambda^{\lambda},\quad \Lambda^d\stackrel{\bL}{\otimes}_{\Gamma^d_K}\Lambda^{\lambda}\cong \Sym^{\lambda},$$
$$\bR\cHom_{\Gamma^d_K}(\Lambda^d,\Lambda^{\lambda})\cong\Gamma^{\lambda},\quad \bR\cHom_{\Gamma^d_K}(\Lambda^d,\Sym^{\lambda})\cong\Lambda^{\lambda},$$
$$\Lambda^d\stackrel{\bL}{\otimes}_{\Gamma^d_K}W_{\lambda}\cong S_{\lambda'},\quad \bR\cHom_{\Gamma^d_K}(\Lambda^d,S_{\lambda})=W_{\lambda'}.$$\end{prop}

We also would like to mention a result on the Serre functor for the category $D^b(\Pol^d_K).$

\begin{prop}(\cite{Kr}, Proposition 5.4) The functor $\Sym^d\stackrel{\bL}{\otimes}-\cong(\Lambda^d\stackrel{\bL}{\otimes}-)^{2}$ is the Serre functor on the category $D^b(\pol^d_K).$\end{prop}

\begin{remark} Strictly speaking, the result about the Serre functor is proved over a field in \cite{Kr}, but the proof over a commutative ring is the same.\end{remark}

So, the inverse Koszul duality functor $\Lambda^d\stackrel{\bL}{\otimes}_{\Gamma^d_K}-$ is the "square root" of the Serre functor $\Sym^d\stackrel{\bL}{\otimes}_{\Gamma^d_K}-.$

\section{Representations of $GL_n$}
\label{ssec:reps_of_GL_n}

The general references for this subsection are \cite{D}, \cite{Gr}.

As above, $K$ denotes the basic commutative ring. Fix some non-negative integer $n\in \Z_{\geq 0},$ and put $G:=GL_n(K).$ We consider $G$ as an algebraic group over $K.$ We denote by $\Rep(K,G)$ the category of $G$-modules which are finitely generated projective over $K.$ We write $\Rep(K,G)^d\subset \Rep(K,G)$ for the full abelian subcategory of representations of degree $d,$ where the degree is taken w.r.t. the central one-dimensional torus $G_m\subset G.$ Clearly, each object $M\in \Rep(K,G)$ splits uniquely as $$M=\bigoplus\limits_{d\in\Z}M_d,\quad M_d\in \Rep(K,G)^d,$$ where all but finitely many $M_d$ equal to zero. Moreover, $$\Hom_{\Rep(K,G)}(P,Q)=0\quad \text{for}\quad P\in \Rep(K,G)^i,\, Q\in \Rep(K,G)^j,\,i\ne j.$$
In other words, $\Rep(K,G)=\bigsqcup\limits_{d\in\Z}\Rep(K,G)^d,$ where the coproduct is taken in the category of small $K$-linear additive categories.

We have a maximal torus $T=G_m^n\subseteq G$ consisting of diagonal matrices. We have a bijection
 $$\Z^n\simeq\Hom(T,G_m),\quad \lambda\leftrightarrow\chi_{\lambda}.$$ Each object $N\in \Rep(K,T)$ can be written uniquely as $$N=\bigoplus\limits_{\lambda\in\Z^n}M_{\lambda}\otimes\chi_{\lambda},\quad M_{\lambda}\in\cP_K,$$
where $M_{\lambda}$ is non-zero for only finitely many of $\lambda.$ We call such $\lambda\in\Z^n$ the weights of $N.$ Clearly, the set of weights is invariant under $\mS_n$-action. Since any representation of $M\in\Rep(K,G)$ can be considered as an object of $\Rep(K,T),$ it also has a finite subset of weights in $\Z^n.$

For $r\in\Z,$ we denote by $$\Rep(K,G)^d_{\geq r}\subseteq \Rep(K,G)^d\quad\text{(resp. } \Rep(K,G)^d_{\leq r}\subseteq \Rep(K,G)^d)$$ the subcategory of representations for which all the weights $\lambda$ satisfy the inequality $\lambda_i\geq r$ (resp. $\lambda_i\leq r$), $1\leq i\leq n.$ We also put $$\Rep(K,G)^d_{[a,b]}:=\Rep(K,G)^d_{\leq b}\cap \Rep(K,G)^d_{\geq a}.$$

The objects of $\Rep(K,G)^d_{\geq 0}$ are known as polynomial representations of degree $d$ \cite{Gr}. Denote by $V_n\in \Rep(K,G)^1$ the tautological representation (of rank $n$). We have a natural exact functor $$\pi_d:\pol^d_K\to \Rep(K,G)^d_{\geq 0},\quad \pi_d(F)=F(V_n).$$ It induces a functor on derived categories $\bR\pi_d:D^b(\pol^d_K)\to D^b(\Rep(K,G)^d_{\geq 0}).$

\begin{defi}We denote by $\cP(m,n)$ the set of Young diagrams $\lambda$ such that $l(\lambda)\leq n,$ $\lambda_1\leq m.$ We denote by $\cP(m,n;d)\subseteq\cP(m,n)$ the subset of Young diagrams $\lambda$ such that $|\lambda|=d.$\end{defi}

 \begin{prop}1) The category $\Rep(K,G)_{\geq 0}^d$ is a highest weight category, which is equivalent to $\Rep(K,S_K(n,d)^{op}).$ The standard (resp. costandard) objects of $\Rep(K,G)_{\geq 0}^d$ are $W_{\lambda}(V_n)$ (resp. $S_{\lambda}(V_n)$), where $|\lambda|=d,$ $l(\lambda)\leq n.$

2) We have natural equivalences of highest weight categories $$\Rep(K,G)_{\geq 0}^d\simeq\pol^d_K[\{W_{\lambda}\}_{|\lambda|=d,l(\lambda)\leq n}],$$
$$\Rep(K,G)_{[0,m]}^d\simeq\Rep(K,G)_{\geq 0}^d[\{W_{\lambda}\}_{\lambda\in\cP(m,n;d)}].$$

3) The functor $\bR\pi_d:D^b(\pol^d_K)\to D^b(\Rep(K,G)^d_{\geq 0})$ is a semi-orthogonal projection, and its kernel has two natural exceptional collections
$$\ker \pi_d=\la W_{\lambda},|\lambda|=d,l(\lambda)>n;\unlhd\ra,\quad \ker \pi_d=\la S_{\lambda},|\lambda|=d,l(\lambda)>n;\unrhd\ra.$$

4) Denote by $i^d_{\geq 0}:D^b(\Rep(K,G)^d_{\geq 0})\to D^b(\pol^d_K)$ (resp. $j^d_{\geq 0}:D^b(\Rep(K,G)^d_{\geq 0})\to D^b(\pol^d_K)$) the left (resp. right) adjoint to $\pi^d.$ Both $i^d_{\geq 0}$ and $j^d_{\geq 0}$ are fully faithful. We have
\begin{equation}\label{i^d(W_lambda(V))}i^d_{\geq 0}(W_{\lambda}(V_n))=W_{\lambda},\quad j^d_{\geq 0}(S_{\lambda}(V_n))=S_{\lambda}\quad\text{for }|\lambda|=d,\,l(\lambda)\leq n.\end{equation}\end{prop}

\begin{proof}1) The equivalence $\Rep(K,G)_{\geq 0}^d\simeq \Rep(K,S_K(n,d)^{op})$ was shown in \cite{Gr} in the case when $K$ is a field, but the proof for commutative rings is the same. Indeed, the category $\Rep(K,G)^d_{\geq 0}$ is equivalent to the category of left comodules over the coalgebra $\Sym^d(\Mat_{n\times n}(K)^*),$ which are finitely generated projective over $K.$ But the coalgebra $\Sym^d(\Mat_{n\times n}(K)^*)$ is a finitely generated free $K$-module, hence a left $\Sym^d(\Mat_{n\times n}(K)^*)$-comodule is the same as left $\Gamma^d(\Mat_{n\times n}(K))$-module. This proves the equivalence.

Description of standard and costandard objects of $\Rep(K,G)_{\geq 0}^d$ follows from 2), which we prove next.

2) By Yoneda lemma, we have an isomorphism of algebras $$S_K(n,d)^{op}\cong\End_{\pol_K^d}(\Gamma^{d,K^n}).$$ According to the discussion after Theorem \ref{th:modules_over_Schur}, the object $\Gamma^{d,K^n}$ generates the same additive Karoubi complete category as $\bigoplus\limits_{|\lambda|=d,l(\lambda)\leq n}\Gamma^{\lambda}.$ Further, by Theorem \ref{th:Littlewood-Richardson} for any Young diagram $\lambda$ with $|\lambda|=d,$ we have a surjection $f:\Gamma^{\lambda}\to W_{\lambda}$ with $\ker(f)\in (\pol_K^d)^{\widetilde{W_{\rhd\lambda}}}.$ By definition (see Proposition \ref{prop:C[Omega]}), we obtain the equality $$\Rep(K,G)_{\geq 0}^d\cong\pol^d_K[\{W_{\lambda}\}_{|\lambda|=d,l(\lambda)\leq n}].$$

Further, the subcategory $\Rep(K,G)^d_{[0,m]}\subseteq\Rep(K,G)^d_{\geq 0}$ is exactly the right orthogonal to $\{\Gamma^{\lambda}(V_n)\}_{\lambda_1>m}.$ Hence, by definition (see Proposition \ref{prop:C[Gamma]}), we obtain the equality
$$\Rep(K,G)_{[0,m]}^d=\Rep(K,G)^d_{\geq 0}[\{W_{\lambda}\}_{\lambda\in\cP(m,n;d)}].$$

3) and 4) then follow from Proposition \ref{prop:C[Omega]}. If we put $\Omega:=\{W_{\lambda}\mid |\lambda|=d,l(\lambda)\leq n\},$ then in the notation of Proposition \ref{prop:C[Omega]} we have $\pi_d=\pi_{\Omega},$ $i^d_{\geq 0}=i_{\Omega},$ $j^d_{\geq 0}=j_{\Omega}.$
\end{proof}

Denote by $i^d_{[0,m]}$ (resp. $j^d_{[0,m]}$) the restriction of $i^d_{\geq 0}$ (resp. $j^d_{\geq 0}$) onto $D^b(\Rep(K,G)^d_{[0,m]}).$

\begin{prop}\label{prop:im_i_im_j}We have \begin{equation}\label{im_i}\im(i^d_{[0,m]})=\la W_{\lambda},\lambda\in\cP(m,n;d);\unlhd\ra.\end{equation}
Similarly,
\begin{equation}\label{im_j}\im(j^d_{[0,m]})=\la S_{\lambda},\lambda\in\cP(m,n;d);\unrhd\ra.\end{equation}\end{prop}

\begin{proof}This follows immediately from \eqref{i^d(W_lambda(V))}.\end{proof}

Now recall the inverse Koszul duality functor from Definition \ref{def:koszul}:
$$\Lambda^d\stackrel{\bL}{\otimes}-:D^b(\pol^d_K)\to D^b(\pol^d_K).$$

\begin{prop}\label{prop:koszul_for_GL}We have an equality of full subcategories of $D^b(\pol^d_K):$
$$\Lambda^d\stackrel{\bL}{\otimes}_{\Gamma^d_K}i^d_{[0,m]}(D^b(\Rep(K,GL_n)^d_{[0,m]}))=j^d_{[0,n]}(D^b(GL_{m}\text{-mod}^d_{[0,n]})).$$\end{prop}

\begin{proof}Indeed, by Proposition \ref{prop:properties_of_koszul} we have that $$\Lambda^d\stackrel{\bL}{\otimes}_{\Gamma^d_K}W_{\lambda}\cong S_{\lambda'}$$ for any Young diagram $\lambda$ with $|\lambda|=d.$ Hence, by Proposition \ref{prop:im_i_im_j} we have
\begin{multline*}\Lambda^d\stackrel{\bL}{\otimes}_{\Gamma^d_K}i^d_{[0,m]}(D^b(\Rep(K,GL_n)^d_{[0,m]}))=\la \Lambda^d\stackrel{\bL}{\otimes}_{\Gamma^d_K}W_{\lambda},\lambda\in\cP(m,n;d)\ra\\
=\la S_{\mu},\mu\in\cP(n,m;d)\ra=j^d_{[0,n]}(D^b(GL_{m}\text{-mod}^d_{[0,n]})).\end{multline*}
This proves the proposition.\end{proof}

\begin{defi}\label{def:koszul_for_GL}We denote by
\begin{equation}\label{koszul_for_GL}\Lambda_{n,m}^d:D^b(\Rep(K,GL_n)^d_{[0,m]})\stackrel{\sim}{\to} D^b(\Rep(K,GL_m)^d_{[0,n]}).\end{equation}
an equivalence such that
$$j^d_{[0,n]}\circ\Lambda^d_{n,m}\cong (\Lambda^d\stackrel{\bL}{\otimes}_{\Gamma^d_K}-)\circ i^d_{[0,m]}.$$
Such an equivalence exists (and is unique up to a natural isomorphism) by Proposition \ref{prop:koszul_for_GL}.\end{defi}

It is clear from the proof of Proposition \ref{prop:koszul_for_GL} that
\begin{equation}\label{Lambda^d_n,m(W_lambda(V_n))}\Lambda^d_{n,m}(W_{\lambda}(V_n))\cong S_{\lambda'}(V_m).\end{equation}

\begin{defi}Denote by $E(m,n;d)\in \Rep(K,GL_n)^d_{[0,m]}$ the following representation:
$$E(m,n;d):=\bigoplus\limits_{\lambda\in\cP(m,n;d)}\Lambda^{\lambda'}(V_n).$$\end{defi}

\begin{prop}\label{prop:tilting_E(m,n;d)}1) The $GL_n$-module $E(m,n;d)$ is a tilting object in $D^b(\Rep(K,GL_n)^d_{[0,m]}).$

2) The algebra $\End(E(m,n;d))$ is split quasi-hereditary. The standard objects in $\Rep(K,\End(E(m,n;d)))$ correspond to $S_{\lambda}(V_n),$ $\lambda\in\cP(m,n;d),$ under the equivalence
$$D^b(\Rep(K,\End(E(m,n;d))))\simeq D^b(\Rep(K,GL_n)_{[0,m]}^d).$$
\end{prop}

\begin{proof} To prove both 1) and 2), it suffices to check that $(\Lambda^d_{m,n})^{-1}(E(m,n;d))\in D^b(\Rep(K,GL_m)_{[0,n]}^d)$ is actually a projective generator in $\Rep(K,GL_m)_{[0,n]}^d.$

Denote by $$\pi:D^b(\Rep(K,GL_m)^d_{\geq 0})\to D^b(\Rep(K,GL_m)_{[0,n]}^d)$$ the left adjoint to the inclusion. Then $\pi$ takes projective objects of $\Rep(K,GL_m)^d_{\geq 0})$ to projective objects of $\Rep(K,GL_m)_{[0,n]}^d,$ and any projective generator is mapped to a projective generator. It remains to note that $$\Lambda^d_{m,n}(\pi(\Gamma^{\mu}(V_m)))=\Lambda^{\mu^{\prime}}(V_n),\quad \mu\in\cP(n,m;d),$$
$$\pi(\Gamma^{\mu}(V_m))=0\text{ if }\mu_1>n,$$
and $\bigoplus\limits_{|\mu|=d,l(\mu)\leq m}\Gamma^{\mu}(V_m)$ is a projective generator of $\Rep(K,GL_m)_{\geq 0}^d.$
\end{proof}

\section{Base change}
\label{sec:base_change}

Let $X$ be a quasi-compact quasi-separated scheme, flat over a commutative ring $K.$ For any homomorphism $K\to K',$
we put
$$X\times_K K'=X\times_{\Spec(K)}\Spec(K').$$
This is again a quasi-compact quasi-separated scheme, flat over $K'.$

\begin{prop}\label{prop:base_change}1) We have a natural $K'$-Morita equivalence $$\Perf(X)\stackrel{\bL}{\otimes}_K K'\simeq \Perf(X\times_K K').$$

2) Suppose that we have a semi-orthogonal decomposition $\Perf(X)=\la\cT_1,\dots,\cT_n\ra.$ If we denote by $\cT_i'\subset\Perf(X\times_K K')$ the full thick triangulated subcategory generated by $\cT_i\stackrel{\bL}{\otimes}_K K',$ then we have a semi-orthogonal decomposition $\Perf(X\times_K K')=\la\cT_1',\dots,\cT_n'\ra.$

Suppose that $X$ is smooth and proper over $\Spec K.$

3) If we have a full exceptional collection $\Perf(X)=\la E_1,\dots,E_n\ra,$ then  we also have a full exceptional collection $\Perf(X\times_K K')=\la E_1',\dots,E_n'\ra,$  where $E_i':=E_i\stackrel{\bL}{\otimes}_K K'.$ Moreover,
$$\bR\Hom(E_i',E_j')\cong\bR\Hom(E_i,E_j)\stackrel{\bL}{\otimes}_K K'.$$

4) Suppose that $T\in \Perf(X)$ is a tilting object. Then $T':=T\stackrel{\bL}{\otimes}_K K'\in\Perf(X')$ is also a tilting object.
\end{prop}

\begin{proof}1) is standard, and 2),3),4) follows immediately from 1).\end{proof}

\section{Derived categories of Grassmannians}
\label{sec:D^b(Grassmannians)}

In this section the basic ring will always be $\Z,$ the ring of integers. All the results generalize immediately to an arbitrary commutative ring (in particular to an arbitrary field) by Proposition \ref{prop:base_change}.

\subsection{Semi-orthogonal decomposition of $D^b(\Gr(k,n))$}
\label{ssec:SOD_of_Gr}

Take some positive integers $0<k< n.$ Let $V\cong\Z^n$ be a free finitely generated $\Z$-module of rank $n,$ and $X=\Gr(k,V)$ be the "Grassmannian of $k$-dimensional vector subspaces in $V$". To be more precise, for any commutative ring $R,$  $X(R)=\Hom(\Spec R,X)$ is the set of $R$-submodules $P\subset R\otimes V,$ such that the $R$-module $(R\otimes V)/P$ is projective of constant rank $n-k$ (hence $P$ is projective of constant rank $k$).

As above, we denote by $V_k$ the tautological representation of $G=GL_k.$ We denote by $W$ the affine space associated to the $\Z$-module $\Hom(V_k,V)=V_k^{\vee}\otimes V.$ That is, $W=\Spec(\Sym^*(V_k\otimes V^{\vee})).$ Then $W$ is acted on by $G.$

We denote by $W^{ss}\subset W$ the open subscheme of the rank $k$ homomorphisms. That is, for any commutative ring $R,$ the set $W^{ss}$ is the set of split injections of $R$-modules $R\otimes V_k\to R\otimes V.$ We have an obvious identification
$$X\cong W^{ss}//G.$$ The complement $W\setminus W^{ss}$ has a natural stratification by the rank function:
 $$W\setminus W^{ss}=Y_0\sqcup Y_1\sqcup\dots\sqcup Y_{k-1},$$ where for each $0\leq r\leq k,$ $Y_r\subset W$ is a locally closed subscheme of homomorphisms of rank $r.$ In particular, $Y_k=W^{ss}.$ We put
 $$W_{\geq r}:=W\setminus \overline{Y_{r-1}}\subseteq W,\quad 0\leq r\leq k,$$
 where $Y_{-1}=\emptyset.$
 We denote by $\iota_r:Y_r\hookrightarrow W_{\geq r}$ the tautological closed embedding. For a closed embedding $Z\hookrightarrow Y$ of smooth schemes, we denote by $\cN_{Z\mid Y}$ the normal bundle.

We have a natural functor $$\Phi:\Rep(\Z,G)\to \Coh(X),$$ which is induced by our quotient presentation of $X.$ More precisely, $\Phi(N)$ corresponds to $N\otimes\cO_{W^{ss}}$ under the equivalence
$$\Coh(X)\simeq \Coh_G(W^{ss}).$$

If $Y$ is a variety with a trivial $G$-action, then we define the subcategories
$$\Coh_{G}(Y)^d,\,\Coh_{G}(Y)^d_{\leq l},\dots\subset\Coh_G(Y).$$
in the same way as for representations of $G,$ see Subsection \ref{ssec:reps_of_GL_n}.

We denote by $\cF$ the tautological subbundle of rank $k$ on $X,$ and by $\cQ$ the tautological quotient bundle of rank $n-k$ on $X.$ We have a short exact sequence
$$0\to \cF\to V\otimes\cO_X\to \cQ\to 0.$$

\begin{lemma}\label{lem:vanishing_of_local_H^*}For any $M\in \Rep(\Z,G)_{\leq (n-k)},$ we have that $H^{\bullet}_{G,Y_r}(M\otimes \cO_{W_{\geq r}})=0$ for $r<k.$\end{lemma}

\begin{proof}The argument is similar to the techniques in \cite{BFK}, although a bit more complicated since we are working over integers instead of a field of characteristic zero.

Since $$\bR\Gamma_{Y_r}(-)\cong\hocolim_{\substack{n\geq 1}}\bR\cHom_{\cO_{W_{\geq r}}}(\cO_{W_{\geq r}}/I_{Y_r}^n,-),$$
it suffices to check that
\begin{equation}\label{O/I^n}H^{\bullet}_G(\bR\cHom_{\cO_{W_{\geq r}}}(\cO_{W_{\geq r}}/I_{Y_r}^n,M\otimes\cO_{W_{\geq r}}))=0,\quad n\geq 1.\end{equation}
Now, the sheaf $\cO_{W_{\geq r}}/I_{Y_r}^n$ has a filtration
$$\cO_{W_{\geq r}}/I_{Y_r}^n\supset I_{Y_r}/I_{Y_r}^n\supset\dots I_{Y_r}^{n-1}/I_{Y_r}^n,$$
with subquotients $\iota_{r*}(\Sym^i(\cN_{Y_r\mid W_{\geq r}}^{\vee})),$ $0\leq i\leq n-1.$ We have
\begin{multline*}\bR\cHom_{\cO_{W_{\geq r}}}(\iota_{r*}(\Sym^i(\cN_{Y_r\mid W_{\geq r}}^{\vee})),\cO_{W_{\geq r}})\cong\\ \iota_{r*}(\Gamma^i(\cN_{Y_r\mid W_{\geq r}})\otimes \det(\cN_{Y_r\mid W_{\geq r}}))[-(n-r)(k-r)].\end{multline*}
Hence, the vanishing \eqref{O/I^n} would follow from
\begin{equation}\label{Gamma^i(N)}H^{\bullet}_G(Y_r,M\otimes\Gamma^i(\cN_{Y_r\mid W_{\geq r}})\otimes \det(\cN_{Y_r\mid W_{\geq r}}))=0,\quad i\geq 0.\end{equation}

Let us now fix some bases for $V$ and $V_k,$ i.e. isomorphisms $V\cong \Z^n,$ $V_k\cong\Z^k.$ Then $W$ is identified with affine space of matrices $\Mat_{n\times k}.$ Denote by $Z_r\subset Y_r$ the subscheme of matrices $B\in\Mat_{n\times k}$ such that $B_{ij}=0$ for $r+1\leq j\leq k.$ Further, take the parabolic subgroup
$$P_r=\{A\in G\mid A_{ij}=0\text{ for }r+1\leq i\leq k,\,1\leq j\leq r\}\subset G.$$ Clearly, $Z_r$ is stable under $P_r.$ We have an obvious $G$-equivariant identification $$G\stackrel{P_r}{\times}Z_r\cong Y_r.$$ Here $G\stackrel{P_r}{\times}Z_r=(G\times Z_r)/P_r,$ where $P_r$ acts on $G$ by right translation. It follows that
\begin{multline*}H^{\bullet}_G(Y_r,M\otimes\Gamma^i(\cN_{Y_r\mid W_{\geq r}})\otimes \det(\cN_{Y_r\mid W_{\geq r}}))\cong\\ H^{\bullet}_{P_r}(Z_r,M\otimes(\Gamma^i(\cN_{Y_r\mid W_{\geq r}})\otimes \det(\cN_{Y_r\mid W_{\geq r}}))_{\mid Z_r})\end{multline*}
for $i\geq 0.$
Hence, the vanishing \eqref{Gamma^i(N)} would follow from
\begin{equation}\label{H^*_P_r}H^{\bullet}_{P_r}(Z_r,M\otimes (\Gamma^i(\cN_{Y_r\mid W_{\geq r}})\otimes \det(\cN_{Y_r\mid W_{\geq r}}))_{\mid Z_r})=0,\quad i\geq 0.\end{equation}
Take the normal subgroup
$$Q_r=\{A\in P_r\mid A_{ij}=\delta_{ij}\text{ for }1\leq i,j\leq r\}\lhd P_r.$$
Clearly, $P_r/Q_r\cong GL_r.$
Note that the action of $Q_r$ on $Z_r$ is trivial. Hence, for any $P_r$-equivariant coherent sheaf $\cE$ on $Z_r$ we have a Hochschild-Serre spectral sequence
$$E_2^{p,q}=H^p_{GL_r}(Z_r,H^q(Q_r,\cE))\Rightarrow H^{p+q}_{P_r}(Z_r,\cE).$$
Hence, the vanishing \eqref{H^*_P_r} would follow from
\begin{equation}\label{H^*(Q_r,-)}H^{\bullet}(Q_r,M\otimes (\Gamma^i(\cN_{Y_r\mid W_{\geq r}})\otimes \det(\cN_{Y_r\mid W_{\geq r}}))_{\mid Z_r})=0,\quad i\geq 0.\end{equation}
Take the one-parameter subgroup $$\lambda_r:G_m\to Q_r,\quad \lambda_r(t)_{ij}=\begin{cases}\delta_{ij} & \text{for }1\leq i\leq r;\\
t\delta_{ij} & \text{for }r+1\leq i\leq k.\end{cases}$$ Also, denote by $U_r$ the unipotent radical of $Q_r.$ Clearly,
$$U_r=\{A\in Q_r\mid A_{ij}=\delta_{ij}\text{ for }r+1\leq i,j\leq k\},\quad Q_r/U_r\cong GL_{k-r}.$$
For any $Q_r$-equivariant coherent sheaf $\cE$ on $Z_r,$ we have the Hochschild-Serre spectral sequence
\begin{equation*}\label{E_2_for_U_r}E_2^{p,q}=H^p(GL_{k-r},H^q(U_r,\cE))\Rightarrow H^{p+q}(Q_r,\cE).\end{equation*}
Hence, to prove the vanishing \eqref{H^*(Q_r,-)}, it suffices to show that
\begin{equation}\label{weights_for_H^*(U_r,-)}H^{\bullet}(U_r,M\otimes (\Gamma^i(\cN_{Y_r\mid W_{\geq r}})\otimes \det(\cN_{Y_r\mid W_{\geq r}}))_{\mid Z_r})\in\Coh_{GL_{k-r}}(Z_r)^{<0},\quad i\geq 0.\end{equation}
To prove \eqref{weights_for_H^*(U_r,-)}, let us first note that $$(\cN_{Y_r\mid W_{\geq r}})_{\mid Z_r}\in\Coh_{GL_{k-r}}(Z_r)^{-1},\quad \rk((\cN_{Y_r\mid W_{\geq r}})_{\mid Z_r})=(n-r)(k-r).$$
It follows that
\begin{equation}\label{weights_for_Gamma^i(N)otimes_det(N)}(\Gamma^i(\cN_{Y_r\mid W_{\geq r}})\otimes \det(\cN_{Y_r\mid W_{\geq r}}))_{\mid Z_r}\in \Coh_{GL_{k-r}}(Z_r)^{-i-(n-r)(k-r)}.\end{equation}
Since $M\in \Rep(\Z,G)_{\leq (n-k)},$ it follows that
\begin{equation}\label{weights_for_F_otimes_O_Z_r}M\otimes\cO_{Z_r}\in \Coh_{GL_{k-r}}(Z_r)^{\leq(n-k)(k-r)}.\end{equation}
Combining \eqref{weights_for_Gamma^i(N)otimes_det(N)} and \eqref{weights_for_F_otimes_O_Z_r}, we get
\begin{multline}\label{weights_for_F_otimes_Gamma_etc}M\otimes (\Gamma^i(\cN_{Y_r\mid W_{\geq r}})\otimes \det(\cN_{Y_r\mid W_{\geq r}}))_{\mid Z_r}\in \Coh_{GL_{k-r}}(Z_r)^{\leq(-i-(k-r)^2)}\\\subset \Coh_{GL_{k-r}}(Z_r)^{<0}.\end{multline}
Now, let us note that $$\cO(U_r)\in GL_{k-r}\text{-Mod}^{\leq 0}.$$ Computing $H^{\bullet}(U_r,-)$ via the cobar resolution, we get
\begin{multline*}H^{\bullet}(U_r,M\otimes (\Gamma^i(\cN_{Y_r\mid W_{\geq r}})\otimes \det(\cN_{Y_r\mid W_{\geq r}}))_{\mid Z_r})\in \Coh_{GL_{k-r}}(Z_r)^{\leq(-i-(k-r)^2)}\\\subset \Coh_{GL_{k-r}}(Z_r)^{<0}.\end{multline*}
This proves \eqref{weights_for_H^*(U_r,-)}. Lemma is proved.
\end{proof}

\begin{lemma}\label{lem:vanishing}For any representation $M\in \Rep(\Z,G)_{\leq (n-k)}$ we have $$H^{\bullet}(X,\Phi(M))=\bigoplus\limits_{d\geq 0}H^{\bullet}(G,M\otimes \Sym^d(V_k\otimes V^{\vee})).$$
\end{lemma}

\begin{proof}We have a Cousin-Grothendieck spectral sequence
$$E_1^{p,q}=H^{p+q}_{G,Y_{-p}}(W_{\geq (-p)},M\otimes \cO_{W_{\geq (-p)}})\Rightarrow H^{p+q}_G(W,M\otimes\cO_W).$$
Lemma \ref{lem:vanishing_of_local_H^*} implies that
$$E_1^{p,q}=0\text{ for }p\ne -k.$$
Hence, we have
$$H^{\bullet}_G(W_{\geq k},M\otimes\cO_{W_{\geq k}})\cong H^{\bullet}_G(W,M\otimes\cO_W).$$
It remains to note that
$$H^{\bullet}_G(W_{\geq k},M\otimes\cO_{W_{\geq k}})=H^{\bullet}(X,\Phi(M)),$$ $$H^{\bullet}_G(W,M\otimes\cO_W)=\bigoplus\limits_{d\geq 0}H^{\bullet}(G,M\otimes \Sym^d(V_k\otimes V^{\vee})).$$
Lemma is proved.\end{proof}

We denote by $$\Phi^d_{[a,b]}:\Rep(\Z,G)_{[a,b]}^d\to \Coh(X)$$ the restriction of $\Phi.$

\begin{lemma}\label{lem:semi-orthogonality}Let $E_1\in \Rep(\Z,G)_{[0,n-k]}^{d_1},$ $E_2\in \Rep(\Z,G)_{[0,n-k]}^{d_2}.$

(i) If $d_1<d_2,$ then $\Ext^{\bullet}(\Phi_{[0,n-k]}^{d_1}(E_1),\Phi_{[0,n-k]}^{d_2}(E_2))=0;$

(ii) If $d_1\geq d_2,$ then $$\Ext^{\bullet}(\Phi_{[0,n-k]}^{d_1}(E_1),\Phi_{[0,n-k]}^{d_2}(E_2))\cong \Ext^{\bullet}_{G}(E_1,\Sym^{d_1-d_2}(V_k\otimes V^{\vee})\otimes E_2).$$
\end{lemma}

\begin{proof}In both cases we have that $E_1^{\vee}\otimes E_2\in \Rep(\Z,G)_{\leq (n-k)}.$ By Lemma \ref{lem:vanishing} we have
\begin{multline}\label{computation_of_ext}\Ext^{\bullet}(\Phi_{[0,n-k]}^{d_1}(E_1),\Phi_{[0,n-k]}^{d_2}(E_2))\cong H^{\bullet}(\Phi(E_1^{\vee}\otimes E_2))\\\cong \bigoplus\limits_{d\geq 0}H^{\bullet}(G,E_1^{\vee}\otimes E_2\otimes \Sym^d(V_k\otimes V^{\vee})).\end{multline}
Further, $\deg(E_1^{\vee}\otimes E_2\otimes \Sym^d(V_k\otimes V^{\vee}))=d+d_2-d_1.$

In the case (i), it follows that the RHS of \eqref{computation_of_ext} vanishes. In the case (ii), it follows that the RHS of \eqref{computation_of_ext}
equals $$H^{\bullet}(G,E_1^{\vee}\otimes E_2\otimes \Sym^{d_1-d_2}(V_k\otimes V^{\vee}))\cong\Ext^{\bullet}_{G}(E_1,\Sym^{d_1-d_2}(V_k\otimes V^{\vee})\otimes E_2).$$
This proves Lemma.\end{proof}

From now on, we denote the extensions of functors $\Phi^d_{[a,b]},\Phi^d_{\geq 0}$ etc. to the bounded derived categories by the same symbols.

\begin{lemma}\label{lem:incl_im_Phi_geq0}We have an inclusion
\begin{equation}\label{incl_im_Phi_geq0}\im(\Phi^d_{\geq 0})\subseteq \la \im(\Phi^0_{[0,n-k]}),\im(\Phi^1_{[0,n-k]}),\dots,\im(\Phi^{d}_{[0,n-k]})\ra\subseteq D^b(X)\end{equation}
for $d\geq 0.$

Moreover, for a Young diagram $lambda$ with $|\lambda|=d,$ $\lambda_1>n-k,$ $l(\lambda)\leq k,$ we have
\begin{equation}\label{incl_w_lambda}W_{\lambda}(\cF)\in \la \im(\Phi^0_{[0,n-k]}),\im(\Phi^1_{[0,n-k]}),\dots,\im(\Phi^{d-1}_{[0,n-k]})\ra.\end{equation}
\end{lemma}

\begin{proof}We prove both \eqref{incl_im_Phi_geq0} and \eqref{incl_w_lambda} by induction on $d.$ For $d=0$ the statement is evident.

Suppose that the statement is proved for $0\leq d\leq m,$ where $m\geq 0.$ We prove it for $d=m+1.$ We first prove
\eqref{incl_w_lambda}. Suppose that $|\lambda|=d,$ $\lambda_1>n-k,$ $l(\lambda)\leq k.$ We have an acyclic Koszul complex
$$0\to \Lambda^0(V)\otimes \Gamma^{\lambda_1}(\cF)\to \Lambda^{1}(V)\otimes \Gamma^{\lambda_1-1}(\cF)\to\dots\to\Lambda^n(V)\otimes \Gamma^{\lambda_1-n}(\cF)\to 0,$$
where we put $\Gamma^l(\cF)=0$ for $l<0.$ Multiplying this complex by $\Gamma^{\lambda_2}(\cF)\otimes\dots \Gamma^{\lambda_{k}}(\cF)$ and applying the inductive assumption, we get
\begin{equation}\label{incl_Gamma^lambda1}\Gamma^{\lambda}(\cF)\in \la \im(\Phi^0_{[0,k]}),\im(\Phi^1_{[0,k]}),\dots,\im(\Phi^{d-1}_{[0,k]})\ra\text{ for }|\lambda|=d,\,\lambda_1>n-k,\,l(\lambda)\leq k.\end{equation}

Further, by Theorem \ref{th:Littlewood-Richardson} 2), $\Gamma^{\lambda}(\cF)$ has a filtration with top subquotient isomorphic to $W_{\lambda}(\cF),$ with all the other subquotients being of the form $W_{\mu}(\cF),$ where $\mu\rhd \lambda,$ $|\mu|=|\lambda|.$ Hence, \eqref{incl_Gamma^lambda1} implies \eqref{incl_w_lambda}.

In the case $|\lambda|=d,$ $\lambda_1\leq n-k,$ $l(\lambda)\leq k,$ we clearly have that $W_{\lambda}(\cF)\in \im(\Phi^{d}_{[0,n-k]}).$
Therefore, for all $\lambda$ with $|\lambda|=d$ we have that
$$W_{\lambda}(\cF)\in \la \im(\Phi^0_{[0,n-k]}),\im(\Phi^1_{[0,n-k]}),\dots,\im(\Phi^{d}_{[0,n-k]})\ra.$$
This proves \eqref{incl_im_Phi_geq0}. Inductive statement is proved.
\end{proof}

We obtain the following result.

\begin{theo}\label{th:SOD_of_Gr}The functors $\Phi^d_{[0,n-k]}:D^b(\Rep(\Z,G)^d_{[0,n-k]})\to D^b(X)$ are fully faithful for $0\leq d\leq k(n-k).$ We have a semi-orthogonal decomposition
$$D^b(X)=\la \im(\Phi^{k(n-k)}_{[0,n-k]}),\im(\Phi^{k(n-k)-1}_{[0,n-k]}),\dots,\im(\Phi^{0}_{[0,n-k]})\ra.$$
Moreover, for $k(n-k)\geq d_1>d_2\geq 0,$ $E_1\in D^b(\Rep(\Z,G)^{d_1}_{[0,n-k]}),$ $E_2\in D^b(\Rep(\Z,G)^{d_2}_{[0,n-k]}),$  we have
\begin{equation}\label{computation_of_RHom}\bR\Hom(\Phi^{d_1}_{[0,n-k]}(E_1),\Phi^{d_2}_{[0,n-k]}(E_2))\cong \bR\Hom_{G}(E_1,\Sym^{d_1-d_2}(V_k\otimes V^{\vee})\otimes E_2)\end{equation}
\end{theo}

\begin{proof}Fullness, faithfulness, semi-orthogonality and \eqref{computation_of_RHom} are implied by Lemma \ref{lem:semi-orthogonality}.

We are left to prove that $D^b(X)$ is generated by $\im(\Phi_{[0,n-k]}).$ Denote by $\cT\subset D^b(X)$ the full thick subcategory generated by $\im(\Phi_{[0,n-k]}).$ It follows from Lemma \ref{lem:incl_im_Phi_geq0} that $\im(\Phi_{\geq 0})\subset \cT.$ In particular, $$\det(\cF)^{\otimes n}\in\cT,\quad n\geq 0.$$ Since $\det(\cF)$ is an anti-ample line bundle, it follows that $\cT=D^b(X).$ This proves theorem.\end{proof}

\begin{remark}The fact that the categories $\im(\Phi^d_{[0,n-k]})$ generate $D^b(X)$ can be shown by the resolution of the diagonal argument, as in \cite{BLVdB} (generalizing the argument of Kapranov \cite{Kap} to the characteristic-free situation).\end{remark}

\subsection{Dual decomposition and Koszul duality}
\label{ssec:dual_decomposition_Koszul}

Now we would like to describe the semi-orthogonal decomposition of $D^b(X)$ which is right dual to the decomposition from Theorem \ref{th:SOD_of_Gr}.
First, consider the dual quotient presentation of Grassmannian. Namely, Let $W'$ be the affine space associated to the $\Z$-module $\Hom(V,V_{n-k}),$ where $V_{n-k}$ is the tautological representation of $GL_{n-k}.$ That is, $W'=\Spec(\Sym^*(V_{n-k}^{\vee}\otimes V)).$ We denote by $W'^{ss}\subset W'$ subscheme of homomorphisms of rank $n-k.$ We have a natural isomorphism
$X\cong W^{\prime ss}//GL_{n-k}.$ Further, we have the induced functor $$\Psi:\Rep(\Z,GL_{n-k})\to \Coh(X),$$ and its restrictions
like $\Psi^d_{[a,b]}$ and so on.

\begin{lemma}\label{lem:incl_im_Psi_geq0} We have
\begin{equation}\label{incl_im_Psi_geq0}\im(\Psi^d_{\geq 0})\in \la \im(\Psi^0_{[0,k]}),\im(\Psi^1_{[0,k]}),\dots,\im(\Psi^{d}_{[0,k]})\ra\end{equation}
for $d\geq 0.$
Moreover, for any Young diagram $\lambda$ with $|\lambda|=d,$ $\lambda_1>k,$ $l(\lambda)\leq n-k,$ we have
\begin{equation}\label{incl_s_lambda}S_{\lambda}(\cQ)\in \la \im(\Psi^0_{[0,k]}),\im(\Psi^1_{[0,k]}),\dots,\im(\Psi^{d-1}_{[0,k]})\ra.\end{equation}
\end{lemma}

\begin{proof}This is completely analogous to (and is formally implied by) Lemma \ref{lem:incl_im_Phi_geq0}.\end{proof}

\begin{theo}\label{th:right_dual_decomposition}1) The functors $\Psi^d_{[0,k]}:D^b(\Rep(\Z,GL_{n-k})^d_{[0,k]})\to D^b(X)$ are fully faithful and we have a semi-orthogonal decomposition
\begin{equation}\label{decomposition_Psi}D^b(X)=\la \im(\Psi^0_{[0,k]}),\im(\Psi^1_{[0,k]}),\dots,\im(\Psi^{k(n-k)}_{[0,k]})\ra,\end{equation}
which is right dual to the decomposition
$$D^b(X)=\la \im(\Phi^{k(n-k)}_{[0,n-k]}),\im(\Phi^{k(n-k)-1}_{[0,n-k]}),\dots,\im(\Phi^{0}_{[0,n-k]})\ra$$
from Theorem \ref{th:SOD_of_Gr}.

2) For each $d,$ the induced equivalence functor $$D^b(\Rep(\Z,GL_k)^d_{[0,n-k]})\to D^b(\Rep(\Z,GL_{n-k})^d_{[0,k]})$$ is naturally isomorphic to $\Lambda^d_{k,n-k}[-d],$ where $\Lambda^d_{k,n-k}$ is the functor introduced in Definition \ref{def:koszul_for_GL}.\end{theo}

\begin{proof}1) Fully-faithfulness of $\Psi^d_{[0,k]},$ as well as semi-orthogonal decomposition \eqref{decomposition_Psi}, is formally implied by Theorem \ref{th:gluing_of_quasi_hereditary_alg}.

To prove duality of decompositions, we need the following auxiliary result.

\begin{lemma}If $p_d:D^b(X)\to \im(\Psi^{d}_{[0,k]})$ denotes the projection onto the component, then the composition
$$(\Psi^{d}_{[0,k]})^{-1}p_d\Psi^d_{\geq 0}:D^b(\Rep(\Z,GL_{n-k})_{\geq 0}^d)\to D^b(\Rep(\Z,GL_{n-k})_{[0,k]}^d)$$ is the right adjoint to the inclusion.\end{lemma}

\begin{proof}By Lemma \ref{lem:incl_im_Psi_geq0}, the composition $(\Psi^{d}_{[0,k]})^{-1}p_d\Psi^d_{\geq 0}$ vanishes on $S_{\lambda}(V_{n-k})$ for
$\lambda_1>k.$ Moreover, it is identical on $D^b(\Rep(\Z,GL_{n-k})_{[0,k]}^d)\subseteq D^b(\Rep(\Z,GL_{n-k})_{\geq 0}^d).$
The assertion follows.
\end{proof}

By Proposition \ref{prop:criterion_for_dual_decomposition}, it suffices to prove that for $0\leq d\leq k(n-k)$ we have
\begin{equation}\label{equality_decompositions}\la \im(\Phi^{d}_{[0,n-k]}),\dots,\im(\Phi^{0}_{[0,n-k]})\ra=\la \im(\Psi^0_{[0,k]}),\dots,\im(\Psi^{d}_{[0,k]})\ra.\end{equation}
We proceed by induction on $d.$

For $d=0,$ we have $\im(\Phi^0_{[0,n-k]})=\im(\Psi^0_{[0,k]})=\la\cO_X\ra.$

Suppose that \eqref{equality_decompositions} is proved for $0\leq d\leq m.$ We prove it for $d=m+1.$
For $1\leq l\leq k$ we have an acyclic Koszul complex
\begin{equation}\label{Koszul_res_for_S^l(N)}0\to\Lambda^l\cF\otimes\Sym^0(V)\to\Lambda^{l-1}\cF\otimes \Sym^1(V)\to\dots\to \Lambda^0\cF\otimes \Sym^l(V)\to \Sym^l(\cQ)\to 0.\end{equation}
In particular, this gives a morphism $\Sym^l(\cQ)[-l]\to \Lambda^l(\cF).$ Further, for any Young diagram $\mu\in\cP(k,n-k;d)$ we can take the tensor product of resolutions \eqref{Koszul_res_for_S^l(N)} for $l=\mu_1,\dots,\mu_{n-k},$ and get a resolution for $\Sym^{\mu}(\cQ).$ We get a morphism
$\Sym^{\mu}[-d]\to\Lambda^{\mu},$ such that its cone is contained in $$\la \im(\Phi^{d-1}_{[0,n-k]}),\im(\Phi^{d-2}_{[0,n-k]}),\dots,\im(\Phi^{0}_{[0,n-k]})\ra=\la \im(\Psi^0_{[0,k]}),\im(\Psi^1_{[0,k]}),\dots,\im(\Psi^{d-1}_{[0,k]})\ra.$$ By Proposition \ref{prop:tilting_E(m,n;d)}, the objects $\Lambda^{\mu}(\cF),$ $\mu\in\cP(k,n-k;d),$ generate $\im(\Phi^{d}_{[0,n-k]}).$ Further, the objects $p_d(\Sym^{\mu}(\cQ)),$ $\mu\in\cP(k,n-k;d),$ generate
$\im(\Psi^{d}_{[0,n-k]}).$ This implies \eqref{equality_decompositions}.

2) For $0\leq d\leq k(n-k),$ put $$F_d:=(\Psi^d_{[0,k]})^{-1}p_d\Phi^d_{[0,n-k]}:D^b(\Rep(\Z,GL_k)^d_{[0,n-k]})\to D^b(\Rep(\Z,GL_{n-k})^d_{[0,k]}).$$ It follows from the proof of 1) that we have natural isomorphisms
$$F_d(\Lambda^{\mu}(V_k))\cong \Lambda^d_{k,n-k}(\Lambda^{\mu}(V_k))[-d],\quad \mu\in\cP(k,n-k;d).$$
It remains to show that both functors $F_d$ and $\Lambda^d_{k,n-k}$ induce the same maps on morphisms between $\Lambda^{\mu}(V_k).$ Since the morphisms form free finitely generated $\Z$-modules, the statement reduces to the case when the basic ring is $\Q$ instead of $\Z.$ But in that case the statement is trivial because the categories $\Rep(\Q,GL_k)$ and $\Rep(\Q,GL_{n-k})$ are semi-simple. This proves theorem.
\end{proof}

\begin{theo}\label{th:exceptional_on_Gr}1) The category $D^b(X)$ has a full exceptional collection $\{S_{\lambda}(\cF)\}_{\lambda\in\cP(n-k,k)}.$ Its right dual exceptional collection is $\{S_{\mu}(\cQ)[-|\mu|]\}_{\mu\in\cP(k,n-k)}.$

2) The category The category $D^b(X)$ has a full exceptional collection $\{W_{\lambda}(\cF)\}_{\lambda\in\cP(n-k,k)}.$ Its right dual exceptional collection is $\{W_{\mu}(\cQ)[-|\mu|]\}_{\mu\in\cP(k,n-k)}.$\end{theo}

\begin{proof}1) Since $\{S_{\lambda}(V_k)\}_{\lambda\in\cP(n-k,k;d)}$ is a full exceptional collection in $D^b(\Rep(\Z,GL_k)^d_{[0,n-k]}),$ it follows from Theorem \ref{th:SOD_of_Gr} that $\{S_{\lambda}(\cF)\}_{\lambda\in\cP(n-k,k)}$ is a full exceptional collection in $D^b(X).$

By Theorem \ref{th:right_dual_decomposition} 1) we have $\bR\Hom(S_{\mu}(\cQ)[-|\mu|],S_{\lambda})=0$ if $|\lambda|\ne|\mu|.$ In the case $|\lambda|=|\mu|,$ by Theorem \ref{th:right_dual_decomposition} 2) we have
\begin{multline*}\bR\Hom(S_{\mu}(\cQ)[-|\mu|],S_{\lambda})\cong
\bR\Hom_{GL_k}((\Lambda^d_{k,n-k})^{-1}(S_{\mu}(V_{n-k})),S_{\lambda}(V_k))\\\cong \bR\Hom_{GL_k}(W_{\mu'}(V_k)),S_{\lambda}(V_k))=\begin{cases}\Z & \text{if }\mu'=\lambda;\\
0 & \text{otherwise.}\end{cases}\end{multline*}

This shows that $\{S_{\mu}(\cQ)[-|\mu|]\}_{\mu\in\cP(k,n-k)}$ is indeed the right dual exceptional collection

2) is analogous.
\end{proof}

\subsection{Tilting vector bundle}
\label{ssec:tilting_vector_bundle}

Let $X=\Gr(k,n)$ be as above.

\begin{defi}We denote by $\cE(k,n)$ the following vector bundle on $X:$
$$\cE(k,n):=\bigoplus_{\lambda\in\cP(n-k,k)}\Lambda^{\lambda'}(\cF).$$\end{defi}

Clearly, we have
\begin{equation}\label{expression_for_cE(k,n)}\cE(k,n)=\bigoplus\limits_{d=0}^{k(n-k)}\Phi^d_{[0,n-k]}(E(n-k,k;d)).\end{equation}

\begin{lemma}\label{lem:higher_exts_vanishing_cE(k,n)}For each Young diagram $\mu\in\cP(n-k,k)$ we have that
\begin{equation}\label{higher_exts_vanishing_cE(k,n)}\Ext^{>0}(\cE(k,n),S_{\mu}(\cF))=0.\end{equation}\end{lemma}

\begin{proof}By \eqref{expression_for_cE(k,n)}, we have to show that
$$\Ext^{>0}(\Phi^d_{[0,n-k]}(E(n-k,k;d)),S_{\mu}(\cF))=0.$$ Put $d':=|\mu|.$ If $d'>d,$ then by Lemma \ref{lem:semi-orthogonality} (i), we have that
$$\Ext^{*}(\Phi^d_{[0,n-k]}(E(n-k,k;d)),S_{\mu}(\cF))=0.$$

Now assume that $d'\leq d.$ Then by Lemma \ref{lem:semi-orthogonality} (ii) we have
\begin{multline}\label{higher_exts_S_mu(cF)}\Ext^{>0}(\Phi^d_{[0,n-k]}(E(n-k,k;d)),S_{\mu}(\cF))\cong\\ \Ext^{>0}_{GL_k}(E(n-k,k;d),\Sym^{d-d'}(V_k\otimes V^{\vee})\otimes S_{\mu}(V_k)).\end{multline}
By Theorem \ref{th:Littlewood-Richardson} 1), the $GL_k$-module $\Sym^{d-d'}(V_k\otimes V^{\vee})\otimes S_{\mu}(V_k)$ is costandardly filtered. Hence, by Proposition \ref{prop:tilting_E(m,n;d)} we have that the RHS of \eqref{higher_exts_S_mu(cF)} equals to zero. This proves lemma.\end{proof}

\begin{theo}\label{th:tilting_cE(k,n)_on_X}The vector bundle $\cE(k,n)$ is a tilting object of $D^b(X).$\end{theo}

\begin{proof}We first show that $\cE(k,n)$ is a generator of $D^b(X).$ Indeed, by Proposition \ref{prop:tilting_E(m,n;d)}, the object $E(n-k,k;d)$ generates $D^b(\Rep(\Z,GL_k)^d_{[0,n-k]}).$ By Theorem \ref{th:SOD_of_Gr} the categories $\im(\Phi^d_{[0,n-k]})$ generate $D^b(X).$ Hence, it follows from \eqref{expression_for_cE(k,n)} that $\cE(k,n)$ generates $D^b(X).$

To show that $\Ext^{>0}(\cE(k,n),\cE(k,n))=0,$ let us note that (by Theorem \ref{th:Littlewood-Richardson} 1)) the object $\cE(k,n)$ has a filtration with subquotients of the form $S_{\mu}(\cF),$ $\mu\in\cP(n-k,k).$ Then, the assertion follows from Lemma \ref{lem:higher_exts_vanishing_cE(k,n)}. Theorem is proved.
\end{proof}

We put
$$\cB(k,n):=\End_{D^b(X)}(\cE(k,n)).$$
Clearly, $\cB(k,n)$ is a finite projective algebra over $\Z.$ By Theorem \ref{th:tilting_cE(k,n)_on_X}, we have a natural equivalence
\begin{equation}\label{equiv_D^b(X)_D^b(Rep)}D^b(X)\cong D^b(\Rep(\Z,\cB(k,n))).\end{equation}

\begin{theo}\label{th:B(k,n)_quasi_hereditary}The algebra $\cB(k,n)$ (resp. the category $\Rep(\Z,\cB(k,n))$) has two natural structures of a split quasi-hereditary algebra (resp. of a highest weight category).

1) In the first structure, the standard (resp. costandard) objects of $\Rep(\Z,\cB(k,n))$ correspond under the equivalence \eqref{equiv_D^b(X)_D^b(Rep)} exactly to $S_{\lambda}(\cF)$ (resp. $S_{\mu}(\cQ)\otimes\omega_X[k(n-k)-|\mu|]$), where $\lambda\in\cP(n-k,k)$ (resp. $\mu\in\cP(k,n-k)$).

2) In the second structure, the standard (resp. costandard) objects of $\Rep(\Z,\cB(k,n))$ correspond under the equivalence \eqref{equiv_D^b(X)_D^b(Rep)} exactly to $W_{\mu}(\cQ)\otimes\omega_X[k(n-k)-|\mu|]$ (resp. $W_{\lambda}(\cF)\otimes\omega_X[k(n-k)]$), where $\mu\in\cP(k,n-k)$ (resp. $\lambda\in\cP(n-k,k)$).\end{theo}

\begin{proof}This result is essentially a straightforward application of Theorem \ref{th:gluing_of_quasi_hereditary_alg}.

We first describe $\cB(k,n)$ as a gluing. For convenience, we denote by $N_d$ the object $E(n-k,k;d)\in \Rep(\Z,GL_k)^d_{[0,n-k]}.$ Let us put $$A_d:=\End_{\Rep(\Z,GL_k)^d_{[0,n-k]}}(N_d),\quad 0\leq d\leq k(n-k).$$
Further, we put
$$M_{d_1,d_2}:=\Hom_{GL_k}(N_{d_1},N_{d_2}\otimes \Sym^{d_1-d_2}(V_k\otimes V^{\vee}))\in\Rep(\Z,A_{d_1}\otimes A_{d_2}^{op}),$$
where $k(n-k)\geq d_1>d_2\geq 0.$ The products
$$\Sym^{d_1-d_2}(V_k\otimes V^{\vee})\otimes\Sym^{d_2-d_3}(V_k\otimes V^{\vee})\to \Sym^{d_1-d_3}(V_k\otimes V^{\vee})$$
induce the morphisms $$m_{d_1,d_2,d_3}:M_{d_2,d_3}\otimes_{A_{d_2}}M_{d_1,d_2}\to M_{d_1,d_3},\quad k(n-k)\geq d_1>d_2>d_3\geq 0,$$
satisfying the associativity condition. By Lemma \ref{lem:semi-orthogonality}, we have a natural isomorphism
$$\cB(k,n)\cong \widetilde{A}:=A_{k(n-k)}\times_M A_{k(n-k)-1}\dots\times_M A_0,$$
where the gluing was introduced in Definition \ref{def:gluing_of_A_i_via_M}.

{\noindent{\bf Claim.}} {\it The algebras $A_i$ and bimodules $M_{ij}$ satisfy the condition $(\star)$ of Theorem \ref{th:gluing_of_quasi_hereditary_alg}.}

\begin{proof}By Proposition \ref{prop:tilting_E(m,n;d)} 2), The standard objects in $\Rep(\Z,A_d)$ correspond to $S_{\lambda}(V),$ $\lambda\in\cP(n-k,k;d),$
under the equivalence $D^b(\Rep(\Z,A_d))\cong D^b(\Rep(\Z,GL_k)^d_{[0,n-k]}).$ Further, we have a commutative diagram of functors
\begin{equation}\label{commutative_diagram}
\begin{CD}
D^b(\Rep(\Z,A_{d_2})) @>-\stackrel{\bL}{\otimes}_{A_{d_2}}M_{d_1,d_2}>> D^b(\Rep(\Z,A_{d_1}))\\
@VVV @VVV\\
D^b(\Rep(\Z,GL_k)^{d_2}_{[0,n-k]}) @>\pi^{d_1}_{[0,n-k]}(-\otimes\Sym^{d_1-d_2}(V_k\otimes V^{\vee}))>> D^b(\Rep(\Z,GL_k)^{d_1}_{[0,n-k]}),
\end{CD}
\end{equation}
where $k(n-k)\geq d_1>d_2\geq 0,$ and the functor
\begin{equation}\label{pi^d_[0,n-k]}\pi^{d}_{[0,n-k]}:D^b(\Rep(\Z,GL_k)^d_{\geq 0})\to D^b(\Rep(\Z,GL_k)^{d}_{[0,n-k]}).\end{equation}
is right adjoint to the inclusion. Further, for each $0\leq d\leq k(n-k),$ By Proposition \ref{prop:tilting_E(m,n;d)} the full subcategory $\Rep(\Z,A_d)^{\widetilde{\Delta}}\subset D^b(\Rep(\Z,A_d))$ corresponds to $(\Rep(\Z,GL_k)^d_{[0,n-k]})^{\widetilde{\nabla}}\subset D^b(\Rep(\Z,GL_k)^d_{[0,n-k]})$ under the natural equivalence.
By Theorem \ref{th:Littlewood-Richardson}, the functor $$-\otimes\Sym^{d_1-d_2}(V_k\otimes V^{\vee}):D^b(\Rep(\Z,GL_k)^{d_2}_{[0,n-k]})\to D^b(\Rep(\Z,GL_k)^{d_1}_{\geq 0})$$ takes $(\Rep(\Z,GL_k)^{d_2}_{[0,n-k]})^{\widetilde{\nabla}}$ to $(\Rep(\Z,GL_k)^{d_1}_{\geq 0})^{\widetilde{\nabla}},$ and the functor \eqref{pi^d_[0,n-k]} for $d=d_1$ takes
$(\Rep(\Z,GL_k)^{d_1}_{\geq 0})^{\widetilde{\nabla}}$ to $(\Rep(\Z,GL_k)^{d_1}_{[0,n-k]})^{\widetilde{\nabla}}.$ Therefore, the commutativity of diagram \eqref{commutative_diagram} implies that the functor $-\stackrel{\bL}{\otimes}_{A_{d_2}}M_{d_1,d_2}$ takes
$\Rep(\Z,A_{d_2})^{\widetilde{\Delta}}$ to $\Rep(\Z,A_{d_1})^{\widetilde{\Delta}}.$ By Lemma \ref{lem:equivalence_of_5_conditions_on_bimodule} this is equivalent to the condition $(\star)$ of Theorem \ref{th:gluing_of_quasi_hereditary_alg}. This proves Claim.
\end{proof}

It follows from Claim that Theorem \ref{th:gluing_of_quasi_hereditary_alg} can be applied to the algebra $\cB(k,n).$

1) We check that the standard $\cB(k,n)$-modules in the first highest weight structure correspond to $S_{\lambda}(\cF).$ Indeed, by Lemma \ref{lem:semi-orthogonality}
$$\bR\Hom(\Phi^d_{[0,n-k]}(N_d),S_{\lambda}(\cF))=\begin{cases}0 & \text{ if }d<|\lambda|;\\
\bR\Hom_{GL_k}(N_d,\Sym^{d-|\lambda|}(V_k\otimes V^{\vee})\otimes S_{\lambda}(V)) & \text{ if }d\geq |\lambda|.\end{cases}$$
But we have that
$$\bR\Hom_{GL_k}(N_d,\Sym^{d-|\lambda|}(V_k\otimes V^{\vee})\otimes S_{\lambda}(V))\cong \Delta(\lambda)\otimes_{A_{|\lambda|}}M_{d,|\lambda|}.$$
Hence the objects $S_{\lambda}(\cF),$ $\lambda\in\cP(n-k,k),$ correspond exactly to the standard objects in the first highest weight structure.

By Theorem \ref{th:exceptional_on_Gr}, the right dual of the full exceptional collection $\{S_{\lambda}(\cF)\}_{\lambda\in\cP(n-k,k)}$ is exactly $\{S_{\mu}(\cQ)[-|\mu|]\}_{\mu\in\cP(k,n-k)}.$ By Serre duality, the left dual collection is $$\{S_{\mu}(\cQ)\otimes\omega_X[k(n-k)-|\mu|]\}_{\mu\in\cP(k,n-k)}.$$ This proves 1).

2) It follows from Theorem \ref{th:gluing_of_quasi_hereditary_alg} that for each $0\leq d\leq k(n-k)$ the (not full!) exceptional collection $\{\Delta_{(2)}(\lambda)\}_{\lambda\in\cP(n-k,k;d)}$ is right dual to $\{\nabla_{(1)}(\lambda)\}_{\lambda\in\cP(n-k,k;d)}.$ We know from 1) that the object $\nabla_{(1)}(\lambda)\in D^b(\Rep(\Z,\cB(k,n)))$ corresponds to $S_{\lambda'}(\cQ)\otimes\omega_X[k(n-k)-|\mu|]\in D^b(X).$ It follows that the object $\Delta_{(2)}(\lambda)\in D^b(\Rep(\Z,\cB(k,n)))$ corresponds to $W_{\lambda'}(\cQ)\otimes\omega_X[k(n-k)-|\mu|]\in D^b(X).$
Finally, it follows from Theorem \ref{th:right_dual_decomposition} that the left dual to the exceptional collection $\{W_{\mu}(\cQ)\otimes\omega_X[k(n-k)-|\mu|]\}_{\mu\in\cP(k,n-k)}$ is exactly $\{W_{\lambda}(\cF)\otimes\omega_X[k(n-k)]\}_{\lambda\in\cP(n-k,k)}.$ This proves theorem.
\end{proof}


\begin{thebibliography}{BLVdB}

\bibitem[BFK]{BFK}M.~Ballard, D.~Favero, L.~Katzarkov, Variation of geometric invariant theory quotients and derived
categories, arXiv:1203.6643 (preprint).


\bibitem[BFS]{BFS}C.P.~Bendel, E.M.~Friedlander, A.~Suslin, Infinitesimal $1$-parameter subgroups
and cohomology, J. Amer. Math. Soc. 10 (1997), 693-728.

\bibitem[Bo]{Bo}G.~Boffi, The universal form of the Littlewood-Richardson rule. Adv. in Math. 68 (1988), no. 1, 40-63.

\bibitem[B]{B}A.~Bondal, Representations of associative algebras and coherent sheaves, (Russian) Izv. Akad. Nauk SSSR Ser.
Mat. 53 (1989), no. 1, 25-44; translation in Math. USSR-Izv. 34 (1990), no. 1, 23-42.

\bibitem[BK]{BK}A.~Bondal, M.~Kapranov, Representable functors, Serre functors, and reconstructions, (Russian) Izv. Akad. Nauk SSSR Ser. Mat. 53 (1989), no. 6, 1183-1205, 1337; translation in Math. USSR-Izv. 35 (1990), no. 3, 519-541.

\bibitem[BVdB]{BVdB}A.~Bondal, M.~Van~den~Bergh, Generators and representability of functors in commutative and noncommutative geometry.  Mosc. Math. J. 3 (2003), no. 1, 1-36.

\bibitem[BLVdB]{BLVdB}R.-O.~Buchweitz, G.J.~Leuschke, M.~Van~den~Bergh, On the derived category of Grassmannians in arbitrary characteristic. Compos. Math. 151 (2015), no. 7, 1242-1264.

\bibitem[CPS1]{CPS}E.~Cline, B.~Parshall, L.~Scott, Finite-dimensional algebras and highest weight categories. J. Reine Angew. Math. 391 (1988), 85-99.

\bibitem[CPS2]{CPS2}E.~Cline, B.~Parshall, L.~Scott, Integral and graded quasi-hereditary algebras. I.
J. Algebra 131 (1990), no. 1, 126-160.

\bibitem[D]{D}S.~Donkin, On tilting modules for algebraic groups.
Math. Z. 212 (1993), no. 1, 39-60.

\bibitem[FS]{FS}E.M.~Friedlander, A.~Suslin. Cohomology of finite group schemes over a field, Invent. Math. 127
(1997), no. 2, 209-270.

\bibitem[Gr]{Gr}J.~A. Green, Polynomial representations of $GL_n,$ Lecture Notes in Mathematics, 830, Springer, Berlin,
1980.


\bibitem[Kap]{Kap}M.M.~Kapranov, On the derived categories of coherent sheaves on some homogeneous spaces. Invent. Math. 92 (1988), no. 3, 479-508.




\bibitem[Kr]{Kr}H.~Krause, Koszul, Ringel and Serre duality for strict polynomial functors. Compos. Math. 149 (2013), no. 6, 996-1018.

\bibitem[LS]{LS}V.~Lunts, O.~Schn\"urer, Smoothness of equivariant derived categories. Proc. Lond. Math. Soc. (3) 108 (2014), no. 5, 1226-1276.




\bibitem[Ro]{Ro}R.~Rouquier,$q$-Schur algebras and complex reflection groups. Mosc. Math. J. 8 (2008), no. 1, 119-158.

\bibitem[S]{S}L.~Scott, Simulating algebraic geometry with algebra. I. The algebraic theory of derived categories. The Arcata Conference on Representations of Finite Groups (Arcata, Calif., 1986), 271-281, Proc. Sympos. Pure Math., 47, Part 2, Amer. Math. Soc., Providence, RI, 1987.

\bibitem[TV]{TV}B.~To\"en, M.~Vaqui\'e, Moduli of objects in dg-categories. Ann. Sci. \'Ecole Norm. Sup. (4) 40 (2007), no. 3, 387-444.


\end{thebibliography}
\end{document}